\documentclass[11pt]{article} 

\usepackage[caption=false]{subfig}
\usepackage[T1]{fontenc}
\usepackage[UKenglish]{babel}
\usepackage[UKenglish]{isodate}
\usepackage[utf8]{inputenc}
\usepackage[overload]{empheq}  
\usepackage[subnum]{cases}  
\usepackage{comment}
\usepackage{graphicx}
\usepackage[space]{grffile} % for spaces in include graphics
\usepackage{epstopdf}
\usepackage{pdfpages}
\usepackage{float} % to write H in the figures
\usepackage{tikz-cd} % diagrams
\usepackage{xcolor}
\RequirePackage[linkcolor = black,citecolor=blue,urlcolor=blue]{hyperref}
\usepackage{amsmath}
\usepackage{cleveref}
\usepackage{amssymb}
\usepackage{color}
\usepackage{amsthm}
\usepackage{booktabs}
\usepackage{bbm}
\usepackage{bm}
\usepackage{caption}
\usepackage{eucal}
\usepackage{float}
\usepackage{latexsym}
\usepackage{multicol}
\usepackage{multirow}
\usepackage{pdfpages}
\usepackage{siunitx}
\usepackage{authblk}
\usepackage{mathtools}

\graphicspath{{Plots/}}

\usepackage{framed}

\usepackage{enumitem}

\newlist{steps}{enumerate}{1}
\setlist[steps, 1]{label = \emph{Step} \arabic*:}

\theoremstyle{plain}
\usepackage{mathrsfs}  
\newtheorem{theorem}{Theorem}[section]

\newtheorem{lemma}[theorem]{Lemma}
\newtheorem{example}{Example}
\newtheorem{remark}[theorem]{Remark}
\newtheorem{assumption}[theorem]{Assumption}
\newtheorem{proposition}[theorem]{Proposition}
\newtheorem{definition}[theorem]{Definition}
\theoremstyle{remark}
\numberwithin{equation}{section}

\begin{document}		
\title{}
\title{Common noise pullback attractors for stochastic dynamical systems.}
\author[]{Federico Graceffa}
\author[]{Jeroen S.W. Lamb}
\affil[]{Department of Mathematics, \\ Imperial College London, \\ 180 Queen's Gate, \\ London SW7 2AZ, \\ United Kingdom}
	
\maketitle

%%%%%%%%%%%%%%%%%%%%%%%%%%%%%%%%%%%%%%%%%%%%%%%%%%%%%%%%%%%%%%%%%%%%%%%%%%%%%%%%%%%%%%%%%%%%%%%%%%%%%%%%%%%%%%
%												Abstract
%%%%%%%%%%%%%%%%%%%%%%%%%%%%%%%%%%%%%%%%%%%%%%%%%%%%%%%%%%%%%%%%%%%%%%%%%%%%%%%%%%%%%%%%%%%%%%%%%%%%%%%%%%%%%%

\begin{abstract}
We consider SDEs driven by two different sources of additive noise, which we refer to as  \emph{intrinsic} and \emph{common}. We establish almost sure existence and uniqueness of pullback attractors with respect to realisations of the common noise only. These \textit{common noise pullback attractors} are smooth probability densities 
that depend only on (the past of) a common noise realisation and to which
the pullback evolution of a corresponding stochastic Fokker-Planck equation converges.
Common noise pullback attractors have a natural motivation in the context of particle systems with intrinsic and common noise, describing the distribution of the system  
conditioned on (the past of) a common noise realisation. 
%a stochastic Fokker-Planck equation describing the evolution of probability distributions driven by the common noise and we establish almost sure convergence of the corresponding pullback evolution.
\end{abstract} 
  
\newpage
\tableofcontents

%%%%%%%%%%%%%%%%%%%%%%%%%%%%%%%%%%%%%%%%%%%%%%%%%%%%%%%%%%%%%%%%%%%%%%%%%%%%%%%%%%%%%%%%%%%%%%%%%%%%%%%%%%%%%%
%								Introduction and summary of the main results
%%%%%%%%%%%%%%%%%%%%%%%%%%%%%%%%%%%%%%%%%%%%%%%%%%%%%%%%%%%%%%%%%%%%%%%%%%%%%%%%%%%%%%%%%%%%%%%%%%%%%%%%%%%%%%
%\include{Intro}
\section{Introduction and summary of the main results}\label{sec:introduction}

%-------------------------------------------------------------------------------------------------------------
%											Motivation
%-------------------------------------------------------------------------------------------------------------
\subsection{Motivation}
In the theory of dynamical systems, broadly speaking, two dominant points of view may be distinguished: the topological point of view (understanding of the dynamics at the level of (typical) individual trajectories) and the probabilistic point of view (understanding of the dynamics at the level of average statistical properties, e.g. through Ergodic Theory) \cite{KatokHasselblatt1993}. Dynamical systems in the presence of noise (such as stochastic dynamical systems defined by SDEs) are predominantly approached from the latter point of view, with  powerful analytical techniques from stochastic analysis and Markov processes \cite{oksendal2003stochastic}. The alternative \textit{random dynamical systems} approach considers dynamical systems with noise as skew-product systems, where noise drives an otherwise deterministic dynamical system and the noise driving process admits a pathwise and probabilistic (in terms of ergodic theory) description. The latter allows a blend of the traditional topological and probabilistic approaches to achieve probabilistic results about the behaviour of trajectories of the (non-autonomous) noise-driven system. For instance, Arnold and co-workers~\cite{arnold1995random} have established the existence of random generalisations of attractors, as well as stable, unstable and centre manifolds.  

In this paper we develop a random dynamical systems point of view for SDEs with two distinguished sources of noise, which we refer to as \textit{intrinsic} and \textit{common} in view of the motivating example of a system with identical non-interacting particles (or agents) that are subject to intrinsic noise at the level of each particle and a common noise that is equal to all. 
Such settings naturally arise in a broad range of applications, for instance in genetics \cite{del2018role,elowitz2002stochastic,swain2002intrinsic,thomas2018population}, neuroscience \cite{abbott2011interactions,lang2010phase}, epidemics \cite{herrerias2017effects}, pattern formation \cite{sagues2007spatiotemporal} and financial mathematics \cite{giles2012stochastic}.

The aim of this paper is to study a stochastic dynamical system with intrinsic and common noise, conditioned on the past of the common noise realisation. In the context of the motivating particle system, this yields a description of the probability distribution of the particle system, the evolution of which is described by a stochastic Fokker-Planck equation, subject to the past of the realisation of the common noise.\footnote{Bressloff calls this the \emph{population} or \emph{SPDE perspective}, in  contrast with the \emph{particle perspective}, where one averages over both the intrinsic and common noise \cite{bressloff2016stochastic}.\label{Bfn}}  We establish the existence and uniqueness of a corresponding  \textit{common noise pullback attractor} for a specific class of SDEs where the intrinsic and common noises are additive Brownian motions. Our approach is not limited to this special class, but the main aim of this paper is to develop the concept of common noise pullback attractors in this specific transparent setting. 

Common noise pullback attractors facilitate the study of time-averaged properties of the distribution describing the particle system, through the application of ergodic theory, cf.~equation \eqref{Perg} in Section~\ref{subsec:common-noise-pullback-attractor}, for example observing the variance of the distribution as a measure of synchronisation. Synchronisation is a widely studied dynamical phenomenon in complex systems with ramifications in a wide range of applications \cite{pikovsky2003synchronization}. In addition, from a modelling perspective, our point of view is natural where the intrinsic noise is inherently or practically not observable, while the common noise can in principle be observed. Examples include the sentiment of traders in markets and voters under the influence of mass media, where the latter can be treated as a stochastic process or as a deterministic signal, leading to the consideration of dynamics with common noise or more general non-autonomous dynamics. In fact, the random dynamical systems approach taken in this paper in principle allows us to address both settings at the same time. 
  
The notion of pullback attractors is well-established in non-autonomous and random dynamical systems, see for instance~\cite{arnold1995random,kloeden2011nonautonomous} and \cite{bates2009random,caraballo2006pathwise,crauel1994attractors,flandoli1996random,lu2019wong,wang2011random,yan2017random} in the context of SPDEs. The analysis of pullback attractors in applications of complex nonlinear systems is gaining popularity in recent years, for instance in the context of climate science and turbulence \cite{drotos2017importance,ghil2008climate}. This paper is a further contribution in this direction. 

%-------------------------------------------------------------------------------------------------------------
%								Common noise pullback attractor
%-------------------------------------------------------------------------------------------------------------
\subsection{Common noise pullback attractor}\label{subsec:common-noise-pullback-attractor}

We consider SDEs in $\mathbb{R}^d$ of the form
\begin{equation}\label{SDE}
d x(t)  = -\nabla V(x(t)) dt + \sigma d W(t) + \eta d B(t),
\end{equation}
where $V$ represents a smooth potential, $\sigma,\eta$ are positive definite matrices, $W$ is a Brownian motion and $B$ represents another source of noise, such as another Brownian motion or related process. 

We refer to $W$ as \emph{intrinsic} and $B$ as \emph{common} noise, motivated by the setting of a system of identical particles with states $x_i\in\mathbb{R}^d$, $i\in\{1,\ldots,N\}$, each subject to an 
intrinsic noise $W_i$ and an identical common noise $B$
\begin{equation}\label{SDEpart}
d x_i(t)  = -\nabla V(x_i(t)) dt + \sigma d W_i(t) + \eta d B(t).
\end{equation}
In the limit of large $N$, the evolution of this particle system is described by the evolution of a measure on $\mathbb{R}^d$. 
This measure has a Lebesgue density $p$ whose evolution is governed by the stochastic Fokker-Planck equation 
\begin{equation}\label{stochasticFP}
d p = \left[ \Delta V(x) p + \nabla V(x) \frac{\partial p}{\partial x} + \frac{1}{2} \sigma^2 \frac{\partial^2 p}{\partial x^2} \right] dt - \eta \frac{\partial p}{\partial x} \circ d B(t),
\end{equation}
where, as usual, $\circ$ refers to the Stratonovich convention of stochastic integration. 

Similar equations have been derived earlier by Giles and Reisinger \cite{giles2012stochastic} in the context of pricing baskets of financial derivatives, Bressloff \cite{bressloff2016stochastic} in the context of neuronal dynamics, Bain and Crisan \cite{bain2009continuous} in the context of stochastic filtering and Carmona and Delarue~\cite{CarmonaDelarue2018} in the context of mean field games.\footnote{Often, \eqref{stochasticFP} is written in It\^{o} form, resulting in an additional term $\eta^2$ added to $\sigma^2$ in the diffusion term.} 

We aim to employ \eqref{stochasticFP}, like Bresloff \cite{bressloff2016stochastic}, to study the evolution of population densities as a function of the common noise. 
%A main difference between our approach and that of Reisinger and Delarue is that the latter does not allow for a pathwise interpretation of the common noise and hence ends up with a point of view where the common noise enhances the diffusion. 
Thereto, we approach the stochastic Fokker-Planck equation \eqref{stochasticFP} from a random dynamical systems point of view, considering the non-autonomous evolution of the density $p$ as a function of the realisation $\beta$  of the common noise $B$. In Section~\ref{sec:non_auto_FP_IVP} and in particular Theorem~\ref{thm:existence_uniqueness_regularity}, we provide a detailed discussion on the existence, uniqueness and regularity of solutions to the initial value problem of \eqref{stochasticFP} as a deterministic non-autonomous PDE, for a sufficiently regular common noise realisation. It is shown that the flow evolves initial conditions in $L^1$ to the Schwartz space of smooth rapidly decaying Lebesgue densities.
In Section~\ref{sec:from_nonauto_to_stoch}, we establish that the SPDE~\eqref{stochasticFP} is a random dynamical system in the sense that it admits a description in skew-product form with ergodic base dynamics generating the common noise $B(t)$, see Lemma~\ref{prop:SFPE_RDS}. 

Traditionally, dynamical systems theory focuses mostly on the long-term behaviour of solutions. In the non-autonomous setting, as it is rare to have convergence in forward time (since the equations of motion vary with time), it is natural to consider the asymptotic behaviour of \emph{pullback dynamics} instead, which has better prospects of convergence and reveals important aspects of the dynamics. Let $\Phi(t,\beta)$ represent the time-$t$ flow of \eqref{stochasticFP} with common noise realisation $\beta$. Instead of studying the behaviour of initial conditions with fixed noise realisation in the limit $t\to\infty$, pullback dynamics considers the behaviour of initial conditions of this flow, fixing the end-time, say at $t=0$, in the limit of the starting time $\tau\to - \infty$. A pullback attractor describes the state of the system, conditioned on the past of the time-dependent input (noise realisation). In the context of particle dynamics with intrinsic and common noise, the objective is to describe the distribution of the particle system with intrinsic noise, subject to the past realisation of the common noise. 

Under a natural assumption on the potential $V$ that guarantees the existence of a unique stationary density for
\eqref{SDE} in the absence of common noise, the main result of this paper is that the stochastic Fokker-Planck equation \eqref{stochasticFP} has a unique pullback attractor that is a random equilibrium, i.e. for almost all common noise realisations $\beta$, the limit
\begin{equation}\label{pbetadef}
p_{\beta}:=\lim_{\tau\to\infty}\Phi(\tau,\theta_{-\tau}\beta)p
\end{equation}
exists in the Schwartz space and is independent of the initial probability density $p\in L^1$, see Theorem~\ref{thm:pullback_attractor}.  
This result relies on the fact, obtained in Section~\ref{sec:contraction_non_auto_FP}, that for almost every noise realisation the non-autonomous evolution is a contraction. 
Moreover, the convergence is uniform in the common noise realisation. We refer to $p_\beta$ as the \textit{common noise} pullback attractor of the SDE \eqref{SDE}. It turns out  that $p_\beta$ is the density of the measure obtained by averaging, for a fixed common noise realisation $\beta$, the canonical $(\omega,\beta)$-dependent pullback measures of SDE \eqref{SDE} over all intrinsic noise relations $\omega$,  cf.~Proposition~\ref{prop:cnpbversusdisintegration}. 

From a random dynamical systems point of view,  $p_\beta$ in \eqref{pbetadef} is called a globally attracting random equilibrium of the stochastic Fokker-Planck equation \eqref{stochasticFP}.  This is the simplest type of attractor one may encounter in a random dynamical system. In general, random (pullback) attractors may display more complicated behaviour, cf.~\cite{crauel1994attractors}. 

Finally, by virtue of ergodicity we find (in Proposition~\ref{prop:erg}) that if $g$ is a continuous observable on the relevant solution space of densities for \eqref{stochasticFP},  $\mathbb{P}_B$-almost surely,
\begin{equation}\label{Perg}
\lim_{\tau\to\infty}\frac{1}{\tau}\int_0^\tau g(\Phi(\tau,\beta) p) dt=\mathbb{E}^{\mathbb{P}_B}[g(p_\cdot)],
\end{equation}
with $\mathbb{E}^{\mathbb{P}_B}$ denoting the expectation with respect to the probability measure $\mathbb{P}_B$ of the common noise. For special types of observables, the expectation \eqref{Perg} is related to an expectation with respect to the stationary measure $\rho$ of the SDE \eqref{SDE}. In particular, when the observable $g$ is a $p_\beta$-expectation of a continuous observable $h:\mathbb{R}^d\to \mathbb{R}$, i.e. $g(p_\beta)=\int_{\mathbb{R}^d} h(x)p_\beta(x)dx$, then
\begin{equation}\label{Pergstat}
\begin{split}
\mathbb{E}^{\mathbb{P}_B}[g(p_\cdot)]&:=\int_{\Omega_B}g(p_\beta)\mathbb{P}_B(d\beta)=
\int_{\Omega_B}\int_{\mathbb{R}^d} h(x)p_\beta(x)dx\mathbb{P}_B(d\beta)\\
&=
\int_{\mathbb{R}^d}h(x)  \int_{\Omega_B}p_\beta(x)\mathbb{P}_B(d\beta)dx=
\int_{\mathbb{R}^d} h(x) \rho(dx)=:\mathbb{E}^\rho[h]
\end{split}
\end{equation}
However, in general the expectation \eqref{Perg} is not expressible in terms of the stationary measure $\rho$ of \eqref{SDE}. For instance, the variance
\[
\mathrm{Var}(p_\beta):=\mathbb{E}^{p_\beta}[x^2]-\left(\mathbb{E}^{p_\beta}[x]\right)^2
\]
which is an indicator of synchronisation (of the particle system), is an observable whose $\mathbb{P}_B$-expectation cannot be deduced from $\rho$, cf.~the examples discussed in Section~\ref{subsec:intro-examples}.

%-------------------------------------------------------------------------------------------------------------
%										Examples
%-------------------------------------------------------------------------------------------------------------
\subsection{Examples}\label{subsec:intro-examples}

%------------------------------ Ornstein-Uhlenbeck process with intrinsic and common noise ------------------%
\subsubsection{Ornstein-Uhlenbeck process with intrinsic and common noise}\label{subsubsec:OU}
Our results are well-illustrated in the elementary example of an Orstein-Uhlenbeck process with intrinsic and common noise 
\begin{equation}\label{linSDE}
d x(t)  = -a x(t)dt + \sigma d W(t) + \eta d B(t),
\end{equation}
with $x\in\mathbb{R}$, $a>0$ and $W$ and $B$ are independent Brownian motions.\footnote{Most of the results here do not require $B$ to be a Brownian motion, cf.~also footnote~\ref{fn:gen-noise}.} Due to the linearity of \eqref{linSDE} the solution of \eqref{stochasticFP} with initial condition $\delta_{x(s)}$  and common noise realisation $\beta$ can be explicitly calculated
(for all $t>s$) to have the form\footnote{For details, see Appendix~\ref{app:explicit_disintegration}.}
\[
p(x,t) = \sqrt{\frac{a}{\pi \sigma^2 (1-e^{-2 a (t-s)})}} \exp \left( - \frac{a}{\sigma^2 (1-e^{-2 a (t-s)})} (x-m_\beta(t,s))^2 \right),
\]
where
\begin{equation*}
m_\beta(t,s) := x(s) e^{-a(t-s)} + \eta \int_s^t e^{-a (t-u)} d \beta(u),
\end{equation*}
and the latter integral is $\mathbb{P}_B$-almost surely finite.
Indeed, by Theorem~\ref{thm:pullback_attractor}, the unique common noise pullback attractor of \eqref{linSDE} is independent of the initial condition (in $L^1$, cf.~Section~\ref{sec:non_auto_FP_IVP}) and thus equals $\mathbb{P}_B$-almost surely the
normal distribution 
\begin{equation}\label{OUcnpb}
p_\beta(x)=\lim_{s\to-\infty}p(x,0)=
%\rho(x,\beta) = 
\sqrt{\frac{a}{\pi \sigma^2}} \exp \left(-\frac{a}{\sigma^2} \left(x-\eta \int_{-\infty}^0 e^{au} d\beta(u) \right)^2 \right).
\end{equation}

This example illustrates how the exact synchronisation of solutions of \eqref{linSDE} in the absence of intrinsic noise ($\sigma=0$) turns into an approximate synchronisation of the corresponding particle system in the presence of small intrinsic  noise ($\sigma\ll 1$), characterized by small $\mathrm{Var}(p_\beta)$. Namely, in the absence of intrinsic noise, $\mathbb{P}_B$-almost surely all pairs of initial conditions $x,y\in\mathbb{R}^d$ pathwise converge, i.e.~$x_\beta(t),~y_\beta(t)$ of \eqref{linSDE} with noise realisation $\beta$ satisfy $\lim_{t\to\infty} |x_\beta(t)-y_\beta(t)|=0$ \cite{crauel1998additive}, while in the presence of intrinsic noise the distribution converges to a normal distribution with variance $\mathrm{Var}(p_\beta)=\frac{\sigma^2}{2a}$.
%, irrespective of the common noise realisation $\beta$. 
The location of this normal distribution depends on (the past of) the common noise realisation $\beta$, i.e.~the mean $m(p_\beta)=m_\beta(0,-\infty)$ and is independent of the intrinsic noise strength $\sigma$. In view of \eqref{Perg}, this implies for the time-averages of the variance and mean of the (particle) distribution that, $\mathbb{P}_B$-almost surely, 
\[
\lim_{\tau\to\infty}\frac{1}{\tau}\int_0^\tau \mathrm{Var}(\Phi(\tau,\beta) p) dt=\frac{\sigma^2}{2a}~\mathrm{and}~
\lim_{\tau\to\infty}\frac{1}{\tau}\int_0^\tau m(\Phi(\tau,\beta) p) dt=\mathbb{E}^\rho(x)=0,
\]
where $\rho$ denotes the stationary measure of \eqref{linSDE}.

We contrast the average of the observed variance along trajectories of \eqref{stochasticFP} with the fact that the stationary density $p_\rho$ of \eqref{linSDE}, $p_\rho(x)=\int p_\beta(x) \mathbb{P}(d\beta)$, has a different variance. 
In particular,  
\[
p_\rho(x)= \sqrt{\frac{a}{\pi (\sigma^2 + \eta^2)}} \exp \left(-\frac{a}{(\eta^2 + \sigma^2)} x^2 \right)
\]
is a normal distribution with mean $0$ and variance $\frac{\sigma^2+\eta^2}{2a}$. 
Indeed, synchronisation of the particle system corresponds to localisation of the pullback measure, rather than to localisation of the stationary measure. If $\sigma$ is small and $\eta$ is large, the particle distribution is asymptotically strongly localized, while the stationary distribution is not.

We note that the Ornstein-Uhlenbeck example \eqref{linSDE} is very special, in particular the fact that the shape of the density $p_\beta$ does not depend on the noise realisation $\beta$. This is a consequence of the linearity of this example, which also yields it exactly solvable. 

%-------------------------------- Bi-stable dynamics with intrinsic and common noise -----------------------%
\subsubsection{Bi-stable dynamics with intrinsic and common noise}\label{subsubsec:double_well}
We next consider the less degenerate, nonlinear, example of \eqref{SDE} with $x\in\mathbb{R}$ and 
$V(x)=\frac{1}{4}x^4-\frac{a}{2}x^2$ is a double-well potential
\begin{equation}\label{dwSDE}
d x(t) = x(t) (a-x(t)^2) dt + \sigma d W(t) + \eta d B(t),
\end{equation}
with $a>0$ and $\sigma,\eta$ constants as above and . $W(t),~B(t)$ are Brownian motions. In this case, the stationary probability density of
the SDE \eqref{dwSDE} admits the explicit expression
\begin{equation}\label{dwSDEstatm}
p_\rho(x) = \frac{1}{N} \exp \left(-\frac{2}{\sigma^2 + \eta^2} \left(\frac{x^4}{4} - a \frac{x^2}{2} \right)\right),
\end{equation}
where $N:=\int \exp \left(-\frac{2}{\sigma^2 + \eta^2} \left(\frac{x^4}{4} - a \frac{x^2}{2} \right)\right)dx$ is a normalization constant

An important difference with \eqref{linSDE} is that \eqref{dwSDE} is nonlinear. To our best knowledge, in this case, the common noise pullback attractor $p_\beta$ of \eqref{dwSDE} does not admit a comprehensive analytical expression, but it can be approximated numerically (for instance, by means of Monte Carlo methods, cf.~\cite{kloeden1992numerical}). 
In Figure~\ref{fig:non_uniform_convergence} some numerically obtained examples of densities for common noise pullback attractors are presented, illustrating how the stationary density (depicted in the background in grey) may differ substantially from the densities of individual pullback attractors $p_\beta$ which depend on the common noise realisation $\beta$. This figure illustrates some of the limitations in dynamical information that a stationary measure of a stochastic dynamical system provides.

\begin{figure}
	\hspace*{-2.1cm}
	\includegraphics[scale=.50]{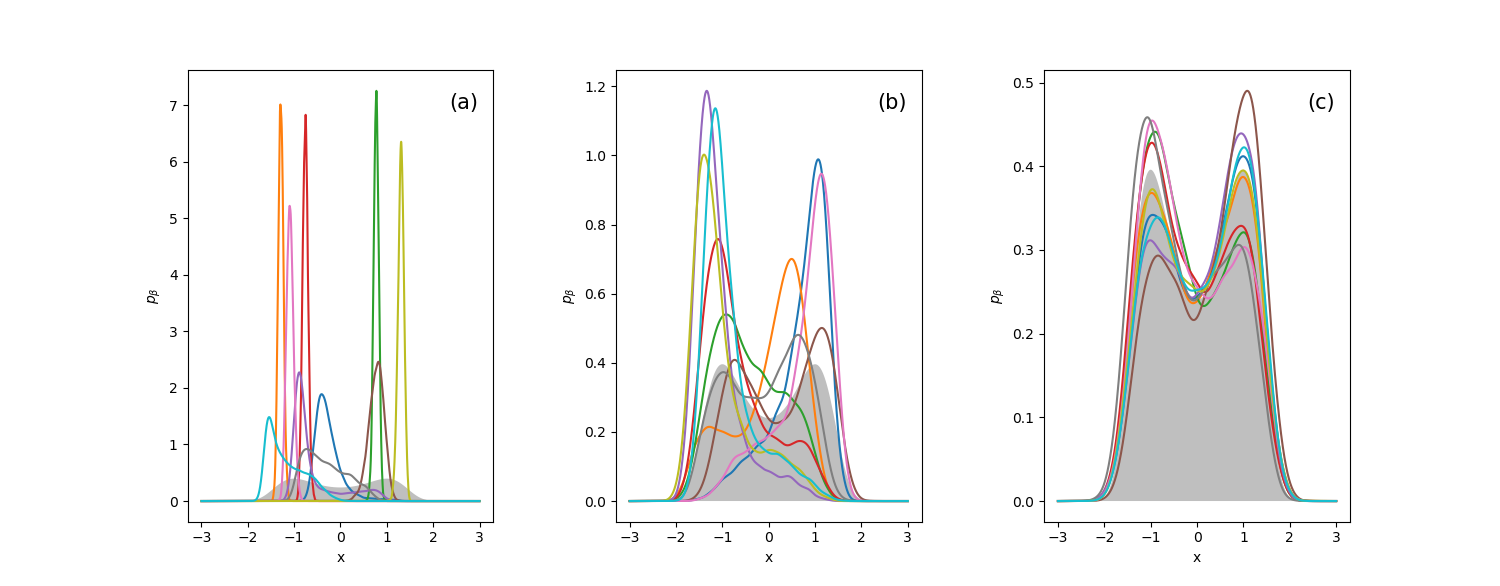}
	\caption{Densities of common noise pullback attractors $p_\beta$ of the SDE with double well potential with intrinsic and common additive noise \eqref{dwSDE}, with $a = 1, \sigma^2 + \eta^2 = 1$ and
(a) $\sigma \ll \eta=0.99$ , (b) $\sigma=\eta=\frac{1}{2}\sqrt{2}$ and (c) $\sigma \gg \eta=0.15$. Pullback attractors for different common noise realisations $\beta$ are represented by graphs with different colors. The stationary density $p_\rho$ \eqref{dwSDEstatm} of the SDE \eqref{dwSDE}, plotted in grey in the background, is identical in all three cases.  Different scales on the $p_\beta$-axis have been chosen so as to achieve similar resolutions in the graphs of the densities of the common noise pullback attractors. When common noise dominates intrinsic noise, $\sigma\ll\eta$ (a), one predominantly observes localized pullback attractors, corresponding to approximate synchronisation. When intrinsic noise dominates common noise, $\sigma\gg\eta$ (c), the densities of the common noise pullback attractors tend to be less localized and relatively close to the stationary density. To obtain an objective quantification of the degree of synchronisation, we have numerically approximated the time-averaged variance of (particle) distributions, as $\mathbb{E}^{\mathbb{P}_B}\left[\mathrm{Var}(p_\cdot)\right]$ by virtue of \eqref{Perg}, yielding the values  (a) 0.04, (b) 0.53 and (c) 0.90 (the latter being close to $\mathrm{Var}(p_\rho)$), in accordance with the perceived degrees of localisation in the density graphs.  
	}
	\label{fig:non_uniform_convergence}
\end{figure}
There are various important questions concerning common noise pullback attractors that we have not addressed here, but which deserve further attention. For instance, it would be of interest to determine the support of the stationary measure of \eqref{stochasticFP}, i.e.~the range of possible densities of common noise pullback attractors, in particular also as a function of system parameters. In the exactly solvable Ornstein-Uhlenbeck example of Section~\ref{subsubsec:OU}, the range consists of a one-parameter family of normal distributions, with identical variance depending on the strength of the intrinsic noise. In the double-well example of Section~\ref{subsubsec:double_well}, Figure~\ref{fig:non_uniform_convergence} suggests that the range is also limited but with a more complicated dependence on system parameters.

%\end{document}

%%%%%%%%%%%%%%%%%%%%%%%%%%%%%%%%%%%%%%%%%%%%%%%%%%%%%%%%%%%%%%%%%%%%%%%%%%%%%%%%%%%%%%%%%%%%%%%%%%%%%%%%%%%%%%%%%
%					The non-autonomous Fokker Planck equation and its initial value problem
%%%%%%%%%%%%%%%%%%%%%%%%%%%%%%%%%%%%%%%%%%%%%%%%%%%%%%%%%%%%%%%%%%%%%%%%%%%%%%%%%%%%%%%%%%%%%%%%%%%%%%%%%%%%%%%%%
\section{The non-autonomous Fokker Planck equation and its initial value problem}\label{sec:non_auto_FP_IVP}
In this section we consider the derivation and analysis of equation \eqref{stochasticFP}, as a non-autonomous Fokker-Planck equation. This forms the basis of our discussion  of \eqref{stochasticFP} in Section~\ref{sec:from_nonauto_to_stoch} in the stochastic setting, as a random dynamical system. 

In Section~\ref{subsec:derivation_non_auto_FP}, we discuss the derivation of the non-autonomous Fokker-Planck equation \eqref{stochasticFP} from two points of view: as the Fokker-Planck equation for a non-autonomous SDE and from a particle system approximation, which motivates the choice of terminology \emph{common noise}. In Section~\ref{subsec:exist_uniq_reg} we establish the existence and uniqueness of the solution for the non-autonomous Fokker-Planck equation~\eqref{stochasticFP} within a suitable setting and discuss how this solution smoothens when $t>0$, given an initial condition $q_0 \in L^1(\mathbb{R}^d)$ at $t=0$. The techniques employed are in principle deterministic and well-established, see eg \cite{clement1987one,evans1997partial,triebel1977interpolation}, but as our specific non-autonomous setting is not normally addressed, we present a self-contained technical discussion in Appendix~\ref{app:non_auto_FP_IVP_proof}. The choice of initial conditions in $L^1$ (rather than in $L^2$, as is commonly found in the literature) is to cater for natural densities relevant to the particle system interpretation, such as Dirac's delta, representing a system with all particles in the same initial state. It turns out that initial conditions in $L^1$ evolve immediately into $L^2$. 

%--------------------------------------------------------------------------------------------------------------
%						Derivation of the non-autonomous Fokker-Planck equation
%--------------------------------------------------------------------------------------------------------------
\subsection{Derivation of the non-autonomous Fokker-Planck equation}\label{subsec:derivation_non_auto_FP}

%-------------------------- The Fokker-Planck equation of the non-autonomous SDE ------------------------------%
\subsubsection{The Fokker-Planck equation of the non-autonomous SDE} 
Let us consider \eqref{SDE} as a non-autonomous SDE where $B(t) = \beta(t)$ is deterministic.  We let $\beta\in C^{1/2}(\mathbb{R}, \mathbb{R}^d)$, defined as the set of functions from $\mathbb{R}$ to $\mathbb{R}^d$ that are everywhere locally $\alpha$-H{\"o}lder continuous for any $\alpha<\frac{1}{2}$.\footnote{The choice of this regularity is motivated by the fact that pathwise realisations of the Brownian motion possess this regularity almost surely, see e.g. \cite{karatzas2014brownian}}. Writing $y(t):=x(t) - \eta \beta(t)$, the SDE~\eqref{SDE} can be written as
\begin{equation}\label{eq:nonauto-SDE}
d y(t) = - \nabla V(y(t)+\eta \beta(t)) dt + \sigma d W(t).
\end{equation}
With $U(y,t;\beta):=V(y(t)+\eta \beta(t))$, we find 
\begin{equation*}
\nabla U(y,t;\beta) =  \nabla V(y+\eta \beta(t)),~~
\Delta U(y,t;\beta) = \Delta V (y+\eta \beta(t))
\end{equation*}
where $\nabla$ and $\Delta$ denote the gradient and Laplacian with respect to the first argument.
The Fokker-Planck equation describing the annealed evolution of Lebesgue probability densities $q$  associated with the SDE~\eqref{eq:nonauto-SDE}  is given by \cite{pavliotis2014stochastic}
\begin{equation}
\partial_t q =  \nabla (\nabla U(y,t;\beta) q) + \frac{1}{2}\sigma^2\Delta q.
\label{eq:non autonomous FP gradient}
\end{equation}
Transforming variables $y$ back to $x$ in \eqref{eq:non autonomous FP gradient}, with densities 
$p(x):=q(y)$, yields \eqref{stochasticFP}.

%--------------------- The Fokker-Planck equation of the non-autonomous particle system -----------------------%
\subsubsection{The Fokker-Planck equation of the non-autonomous particle system}\label{sec:partmot} 

The Fokker-Planck equation~\eqref{stochasticFP} can also be motivated directly from a particle system point of view, following for instance \cite{kurtz1999particle}. Let us consider a system of particles $x_i$, $i=1,\ldots, N$ satisfying \eqref{SDEpart}
\begin{equation*}
d x_i(t) = - \nabla V(x_i(t)) dt + \sigma d W_i(t) + \eta d \beta(t),
\end{equation*}
with $W_i$ independent Brownian motions representing the intrinsic noise and $\beta\in C^{1/2}(\mathbb{R}, \mathbb{R}^d)$ a deterministic common driving. Let $\varphi \in C_b^2(\mathbb{R}^d)$ be an observable which is a bounded function of compact support with bounded first and second derivative. Its evolution is given by
\begin{equation}\label{eq:diff_or_not_diff}
\begin{split}
\varphi(x_i(t)) & = \varphi(x_i(0)) + \int_0^t  \left(- \nabla^T \varphi(x_i(s)) \nabla V(x_i(s)) + \frac{1}{2} \sigma^2 \Delta \varphi (x_i(s)) \right)ds \\
& \quad + \int_0^t \sigma \nabla^T \varphi(x_i(s)) d W_i(s) + \int_0^t \eta \nabla^T \varphi(x_i(s))  d \beta(s).
\end{split}
\end{equation}
It is crucial to recognise that the stochastic integral with respect to the intrinsic noise $W_i$ represents a distribution, while the other integrals yield scalars. The empirical measure for a particle distribution may be defined, as usual, as
\begin{equation*}
\nu(t) := \lim_{N \to +\infty} \frac{1}{N} \sum_{i=1}^N \delta_{x_i(t)},
\end{equation*}
where $\delta_x$ denotes the Dirac measure at $x$ and we consider convergence in the weak topology of the collection of all finite signed Borel probability measures on $\mathbb{R}^d$.\footnote{The Glivenko-Cantelli theorem \cite{varadarajan1958convergence} provides more detail on the convergence in the weak topology of the empirical measure. See Boissard and Le Gouic \cite{boissard2014mean}, for a discussion of the convergence of the empirical measure with respect to the Wasserstein distance.} Integrating both sides of \eqref{eq:diff_or_not_diff} with respect to this empirical measure, denoting $\nu_t(\varphi) :=\int \varphi d\nu_t$, yields the SPDE
\begin{equation}\label{preSPDE}
\nu_t(\varphi) = \nu_0(\varphi) + \int_0^t \nu_s (A_1 \varphi) ds + \int_0^t \nu_s(A_2 \varphi) d \beta(s),
%\label{eq:spde_U}
\end{equation}
where $A_1,A_2$ are the differential operators
\begin{align*}
A_1 \varphi & := \frac{1}{2} \sigma^2 \Delta \varphi  - \nabla^T \varphi \nabla V\\
A_2 \varphi & := \eta \nabla^T \varphi.
\end{align*}
In particular, the term containing the stochastic integral vanishes, see \cite[Proof of Theorem 3.1]{kurtz1999particle}. Following \cite[Chapter 7.3]{bain2009continuous}, assuming that the empirical measure $\nu$ has a sufficiently smooth Lebesgue density $p(t)$, one may reformulate \eqref{preSPDE} as
\begin{equation*}
\begin{split}
\nu_t(\varphi) & = \int_{\mathbb{R}} \varphi(x) p(t,x) dx \\
& = \int_{\mathbb{R}} \varphi(x) \left( p(0,x) + \int_0^t A_1^* p(s,x) ds + \int_0^t A_2^* p(s,x) d \beta(s) \right) dx,
\end{split}
\end{equation*}
where $A_1^*,A_2^*$ are given by
\begin{align*}
A_1^* \psi & = \frac{1}{2} \sigma^2 \Delta \psi + \nabla (\nabla V \psi) = \frac{1}{2} \sigma^2 \Delta \psi + \Delta V \psi + \nabla \psi \nabla V \\
A_2^* \psi & = - \eta \nabla^T \psi.
\end{align*}
We thus obtain
\begin{equation*}
p(t,x) = p(0,x) + \int_0^t A_1^* p(s,x) ds + \int_0^t A_2^* p(s,x) d \beta(s),
\end{equation*}
which is the integral form of~\eqref{stochasticFP}.

%--------------------------------------------------------------------------------------------------------------
%					The initial value problem: existence, uniqueness and regularity
%--------------------------------------------------------------------------------------------------------------
\subsection{The initial value problem: existence, uniqueness and regularity}\label{subsec:exist_uniq_reg}

We denote by $C_c^{\infty}(\mathbb{R}^d)$ the space of all smooth and compactly supported functions on $\mathbb{R}^d$ and  $$\langle q, \varphi \rangle:=\int_{\mathbb{R}^d} \varphi  qdx.$$ We further denote by $\mathcal{P}(\mathbb{R}^d)$ the space of all Borel probability measures on $\mathbb{R}^d$. Our notion of a weak solution is as follows:
\begin{definition}[Weak $L^1$-probability solution]
A function $q \in C([0,T]; L^1(\mathbb{R}^d); C^{1/2}(\mathbb{R}, \mathbb{R}^d))$ is a weak $L^1$-probability solution of the initial value problem for the non-autonomous Fokker-Planck equation \eqref{eq:non autonomous FP gradient} if $q$ solves this equation with initial condition $q_0 \in \mathcal{P}(\mathbb{R}^d)$ and $\beta \in C^{1/2}(\mathbb{R}, \mathbb{R}^d)$, such that $q(t) \in \mathcal{P}(\mathbb{R}^d)$ for all $t>0$ and $\langle q(t), \varphi \rangle \to \langle q_0, \varphi \rangle$ as $t \to 0$ for all $\varphi \in C_c^{\infty}(\mathbb{R}^d)$.
\end{definition} 
 
Our main result establishes the existence and uniqueness of probability solutions of the $L^1$ initial value problem and the fact that such solutions are smooth and their derivatives of any order are rapidly decreasing after any finite time, i.e. they belong to the Schwartz space
\begin{equation*}
\mathcal{S} := \mathcal{S}(\mathbb{R}^d) := \left\{f \in C^{\infty}(\mathbb{R}^d): \forall\ n,m \in \mathbb{N}^d, \|f\|_{n,m} < \infty \right\},
\end{equation*}
where $C^{\infty}(\mathbb{R}^d)$ denotes the space of infinitely differentiable functions on $\mathbb{R}^d$ and \footnote{Here we use the multi-index notation, as in Evans~\cite{evans1998partial}.}
\begin{equation}
\|f\|_{n,m} := \sup_{x \in \mathbb{R}^d} |x^n (D^m f)(x)|.\label{seminorm}
\end{equation}

%-------------------------------  Main Theorem: Existence, uniqueness and regularity --------------------------%
\begin{theorem}[Existence, uniqueness and regularity]\label{thm:existence_uniqueness_regularity}
Let $q_0 \in \mathcal{P}(\mathbb{R}^d)$ and $\beta \in C^{1/2}(\mathbb{R},\mathbb{R}^d)$. Let us assume the potential $U$ is $C^\infty$ in $x$ and satisfies the dissipation condition
\begin{equation}
\nabla U(x,t;\beta) \cdot x \|x\|^2 \geq \frac{1}{2} \|x\|^6 - C |\beta|^6
\label{eq:dissipation_assumption}
\end{equation}
for some constant $C>0$. Then, the non-autonomous Fokker-Planck equation \eqref{eq:non autonomous FP gradient} with initial condition $q(0)=q_0$ admits a unique weak probability solution $q(t) \in \mathcal{P}(\mathbb{R}^d)$ for all $t>0$ which is everywhere locally $\alpha$-H{\"o}lder continuous in time for any $\alpha<\frac{1}{2}$, such that $q(\cdot,t;\beta) \in \mathcal{S}$ for $t>0$.
\end{theorem}
We defer the proof of this result to Appendix~\ref{app:non_auto_FP_IVP_proof}.

It turns out that, $L^1$ solutions are not unique in general. In Appendix~\ref{app:non_auto_FP_IVP_proof}, it is shown that uniqueness is ensured by a weighted integrability condition, expressing the fact that no mass can comes from infinity, nor disappears through infinity, in finite time. It turns out (Lemma~\ref{lem:probability_solution}) that this weighted integrability condition is equivalent to the conservation of mass, hence ensuring existence and uniqueness in the context of probability solutions. In view of the particle systems motivation of Section~\ref{sec:partmot}, this setting is natural as it concerns the conservation of particles.

%------------------------------------- Remark: Dissipation condition  ---------------------------------------%
\begin{remark}%[Dissipation condition]
Note that the dissipation condition \eqref{eq:dissipation_assumption} also ensures the existence of a unique stationary solution of~\eqref{SDE}. This condition is fulfilled, for example, in the case of the double well potential $-\nabla V(x) = x (a-\|x\|^2)$, cf. the example discussed in Section~\ref{subsubsec:double_well}. 
\end{remark}

%%%%%%%%%%%%%%%%%%%%%%%%%%%%%%%%%%%%%%%%%%%%%%%%%%%%%%%%%%%%%%%%%%%%%%%%%%%%%%%%%%%%%%%%%%%%%%%%%%%%%%%%%%%%%
%					Contraction property of the non-autonomous Fokker-Planck equation
%%%%%%%%%%%%%%%%%%%%%%%%%%%%%%%%%%%%%%%%%%%%%%%%%%%%%%%%%%%%%%%%%%%%%%%%%%%%%%%%%%%%%%%%%%%%%%%%%%%%%%%%%%%%%
%\include{pullback_attractor_coupling_methods}

\section{Contraction property of the non-autonomous Fokker-Planck equation}\label{sec:contraction_non_auto_FP}
In this section we establish that the time-$t$ evolution operator $\Phi(t,\beta)$ of the non-autonomous Fokker-Planck equation~\eqref{eq:non autonomous FP gradient} is a contraction for all $t>0$ and $\beta \in C^{1/2}(\mathbb{R},\mathbb{R}^d)$. Technical proofs are deferred to Appendix~\ref{app:contraction_non_auto_FP_proof}. 

In the autonomous setting ($\eta=0$), if $V$ is strictly convex, $\Phi$ is known to be a contraction for all $t>0$ if the metric on the solution space is chosen to be the usual Wasserstein distance $W^p$ for any $p \geq 1$ \cite{bolley2012convergence}. Moreover, strict convexity of $V$ is also a necessary condition \cite{von2005transport}. 
Under the milder assumption that the potential $V$ is strictly convex outside a given ball in $\mathbb{R}^d$, with is a less restrictive and more realistic condition, Eberle \cite{eberle2016reflection} established contractivity of the evolution (again in the autonomous setting) for an appropriately chosen Kantorovich-Rubinstein metric. 
In this section, we adapt the results from \cite{eberle2016reflection} for autonomous Fokker-Planck equations to the non-autonomous setting.

Let us consider again~\eqref{SDE} as a non-autonomous SDE where $B(t) = \beta(t)$ at all times $t$ for some $\beta \in C^{1/2}(\mathbb{R},\mathbb{R}^d)$,
\begin{equation}
d x(t) = - \nabla V(x(t)) dt + \sigma d W(t) + \eta d \beta(t),
\label{eq:non-auto gradient SDE}
\end{equation}
where $W$ is a $d$-dimensional Brownian motion, $\sigma,\eta \in \mathbb{R}^{d \times d}$ are constant matrices with positive determinants and the potential $V$ satisfies the same assumptions as in Theorem~\ref{thm:existence_uniqueness_regularity}. We denote by $\mu_{t,\beta}$ and $\nu_{t,\beta}$ the time-$t$ evolved probability measures with respect to a given input $\beta$ and initial conditions $\mu,\nu$ respectively. In other words, with $p_\mu$ denoting the Lebesgue density of $\mu$, we have
\begin{equation}
\mu_{t,\beta}(A):=\int_A\Phi(t,\beta)p_\mu(x) dx,~\forall A \in\mathscr{B}(\mathbb{R}^d),~t\geq 0,
\label{eq:time_t_evolved_prob}
\end{equation}
and similarly for $\nu_{t,\beta}$. We employ a reflection coupling method to determine a bound for the distance between $\mu_{t,\beta}$ and $\nu_{t,\beta}$ with respect to some appropriately chosen metric. This method entails the introduction of an additional auxiliary process $y(t)$ such that $x(t) = y(t)$ for $t \geq T$, for some $T$, adapting \cite[Eqs. (2)-(3)]{eberle2016reflection} to the non-autonomous setting:
\begin{equation*}
	\begin{cases}
		d y(t) = - \nabla V(y(t)) dt + \sigma (I - 2 e(t) e^\top(t)) d W(t) + \eta d \beta(t), & \text{for}\ t<T \\
		y(t) = x(t), & \text{for}\ t \geq T
	\end{cases}
\end{equation*}
where $T:= \inf \left\{t \geq 0: x(t) = y(t) \right\}$ is the coupling time and $e e^\top$ is the orthogonal projection onto the unit vector
\begin{equation*}
	e(t) := \frac{\sigma^{-1} (x(t) - y(t))}{|\sigma^{-1} (x(t) - y(t))|}.
\end{equation*}
The general aim is to construct a function $f$ such that the process $e^{ct} f(|x(t) - y(t)|)$ is a (local) supermartingale for $t<T$, with a constant $c>0$ that is maximized by choosing $f$ appropriately. This ensures uniform, exponential contraction with respect to a Kantorovich-Rubinstein metric $\mathcal{W}_f$. 

%---------------------------- Definition: Kantorovich-Rubinstein metric ------------------------------------%
\begin{definition}[Kantorovich-Rubinstein distance]\label{def:non-autonomous Kantorovich metric}
	Let $f\in C^2([0,\infty))$ be concave and increasing with  $f(0)=0, f'(0)=1$. The $\mathcal{W}_{f}$-distance between two Borel probability measures $\mu,\nu \in \mathcal{P}(\mathbb{R}^d)$ is defined by
	\begin{equation*}
	\mathcal{W}_{f}(\mu,\nu)  := \inf_{\gamma} \mathbb{E}^\gamma [d_f(x,y)] = \inf_{\gamma} \int d_f(x,y) \gamma(dx dy) 
	\end{equation*}
	with $\mathbb{R}^d$ distance
	$d_f(x,y):= f(\|x-y\|)$, where  $\|\cdot\|$ is a norm on $\mathbb{R}^d$, and the infimum is taken over all couplings $\gamma$ of $\mu$ and $\nu$.\footnote{Recall that a Borel measure $\gamma$ on $X\times X$ is called a coupling of Borel measures $\mu$ and $\nu$ on $X$ if $\gamma(A\times X)=\mu(A)$ and $\gamma(X\times B)=\nu(B)$ for all $A,B\in\mathscr{B}(X)$.}  
\end{definition}
In this paper, we always choose the norm to be $\|\cdot\|=|\sigma^{-1} \cdot |$, with $|\cdot|$ denoting the Euclidean norm in $\mathbb{R}^d$ and $\sigma$ the nondegenerate diffusion matrix from \eqref{eq:non-auto gradient SDE}.

%Typical choices for the norm are the Euclidean norm $\|z\|=|z|$ and the intrinsic norm  $\|z\| = |\sigma^{-1} z|$.  
%An appropriate function $f$ defining the metric $\mathcal{W}_f$ is found using a so-called \emph{reflection coupling} method. 
We adapt \cite[Theorem 1 and Corollary 2]{eberle2016reflection} to obtain:

%---------------------------- Proposition:  Kantorovich-Rubinstein contraction ----------------------------%
\begin{proposition}[Kantorovich-Rubinstein contraction]\label{prop: exp_contraction_kantorovich}
Consider the non-autonomous stochastic differential equation \eqref{eq:non-auto gradient SDE} and the setting of  Theorem~\ref{thm:existence_uniqueness_regularity}. Let $\mu_{t,\beta},\nu_{t,\beta}$ be time-$t$ evolved probability measures, as defined in~\eqref{eq:time_t_evolved_prob}. Then, there exist a constant $c>0$ and a convex and increasing function $f$ such that for any $\beta \in C^{1/2}(\mathbb{R},\mathbb{R}^d)$, $t>0$ and initial conditions $\mu,\nu \in \mathcal{P}(\mathbb{R}^d)$,
\begin{equation*}
\mathcal{W}_{f}(\mu_{t,\beta},\nu_{t,\beta}) \leq e^{-c t} \mathcal{W}_{f}(\mu,\nu).
\end{equation*}
\end{proposition}
Note that the function $f$ in this proposition can be determined constructively. It turns out that convergence in the chosen Kantorovich-Rubinstein metric implies convergence in $L^1$.

%-------------------------------------- Proposition ------------------------------------------------------%
\begin{proposition}[Convergence in $L^1$]\label{prop:convergence_Wf_L1}
Consider the non-autonomous stochastic differential equation \eqref{eq:non-auto gradient SDE} and the setting of  Theorem~\ref{thm:existence_uniqueness_regularity}. Let $\mu_{t,\beta}$ be the time-$t$ evolved probability measure, as defined in~\eqref{eq:time_t_evolved_prob}. Assume that $\beta \in C^{1/2}(\mathbb{R},\mathbb{R}^d)$ is such that $(\mu_{t,\beta})_{t>0}$ is a Cauchy sequence with respect to $\mathcal{W}_f$. Then, the sequence $(p_{t,\beta})_{t>0}$ of the associated Lebesgue densities converges in $L^1$.
\end{proposition}
At this point it is important to note that forward convergence, as obtained in Proposition~\ref{prop: exp_contraction_kantorovich}, does not necessariy imply pullback convergence. %Convergence in a pullback sense means that the state of the system is fully determined by the knowledge of the past. The more we know about the past, the more the convergence stabilizes. 
While the contraction property ensures that all solutions approach each other as time progresses forwards,  in order to guarantee pullback convergence additional conditions (on $\beta$) must be satisfied. For instance, boundedness of  $\beta\in C^{1/2}(\mathbb{R},\mathbb{R}^d)$ would suffice. In Section~\ref{sec:from_nonauto_to_stoch} we consider the stochastic setting of \eqref{stochasticFP} and obtain in Theorem~\ref{thm:pullback_attractor} pullback convergence for $\mathbb{P}_B$-almost all Brownian paths $\beta$. 
%and its detailed proof in Appendix~\ref{app:contraction_non_auto_FP_proof}). \blue{Is this well written?}.

%%%%%%%%%%%%%%%%%%%%%%%%%%%%%%%%%%%%%%%%%%%%%%%%%%%%%%%%%%%%%%%%%%%%%%%%%%%%%%%%%%%%%%%%%%%%%%%%%%%%%%%%%%%%%%%%
%						Stochastic Fokker-Planck equation as a random dynamical system
%%%%%%%%%%%%%%%%%%%%%%%%%%%%%%%%%%%%%%%%%%%%%%%%%%%%%%%%%%%%%%%%%%%%%%%%%%%%%%%%%%%%%%%%%%%%%%%%%%%%%%%%%%%%%%%%
\section{The stochastic Fokker-Planck equation as a random dynamical system}\label{sec:from_nonauto_to_stoch}
In Sections~\ref{sec:non_auto_FP_IVP} and~\ref{sec:contraction_non_auto_FP}, we have considered the Fokker Planck equation~\eqref{stochasticFP} as a non-autonomous PDE. In this Section we consider the stochastic setting with $B(t)$ a Brownian motion and show that the resulting stochastic Fokker-Planck equation \eqref{stochasticFP} is a random dynamical system.\footnote{Our results extend naturally to other common noise processes $B(t)$; for instance, those described by an SDE of the form $d B(t) = f(B(t)) dt + \eta d \tilde{W}(t)$ for some $f \in C^1$ and Brownian motion $\tilde{W}(t)$, such as the Ornstein-Uhlenbeck process.\label{fn:gen-noise}} 
%The analysis of pullback attractors in this setting extends naturally from the 
We establish almost sure pullback attractors, using the contractivity obtained in the non-autonomous analysis in Section %~\ref{sec:non_auto_FP_IVP} and~
\ref{sec:contraction_non_auto_FP}, noting that sample paths of the Brownian motion $B(t)$ are $\mathbb{P}_B$-almost surely in 
$C^{1/2}(\mathbb{R},\mathbb{R}^d)$
%the  key assumption on $\beta(t)$ ($\alpha$-H{\"o}lder continuity for any $\alpha<\frac{1}{2}$) holds for almost all sample paths of $B(t)$ 
\cite{karatzas2014brownian}.

The main results of this section concern the fact that \eqref{stochasticFP} is a random dynamical system (Proposition~\ref{prop:SFPE_RDS}) which (almost surely) possesses a unique pullback attractor (Theorem~\ref{thm:pullback_attractor}) and the correspondence between the common noise pullback attractor of \eqref{stochasticFP} and a partial disintegration of the invariant Markov measure of \eqref{SDE} (Proposition~\ref{prop:cnpbversusdisintegration}). Technical proofs are deferred to Appendix~\ref{app:from_nonauto_to_stoch_proof}. We first recall briefly some preliminaries from the random dynamical system approach towards stochastic differential equations, involving the sample path space of Brownian motions \cite[Chapters 1,2 and Appendix A]{arnold1995random}. 

Let $(\Omega,\mathscr{F},\mathbb{P})$ be a probability space and $X$ be a metric space with Borel $\sigma$-algebra $\mathscr{B}(X)$. We consider the situation with two-sided continuous time $t\in\mathbb{R}$. A \emph{random dynamical system} on $X$ consists of two components. The first is a \emph{metric dynamical system} modelling the noise. This is a $(\mathscr{B}(\mathbb{R}) \otimes \mathscr{F}, \mathscr{F})$-measurable function $\theta:\mathbb{R} \times \Omega \to \Omega$ such that 
\begin{enumerate}[label=(\roman*)]
	\setlength{\itemsep}{0pt}
	\setlength{\parskip}{0pt}
	\item $\theta(0,\omega) = \omega$ and $\theta(t+s,\omega) = \theta(t,\theta(s,\omega))$ for all $t,s \in \mathbb{R}, \omega \in \Omega$,
	\item the measure is preserved, i.e. $\mathbb{P}(\theta(t,A)) = \mathbb{P}(A)$ for all $t \in \mathbb{R}$ and $A \in \mathscr{F}$.
\end{enumerate}
Moreover,  $\theta$ is said to be \emph{ergodic} if for any $A \in \mathscr{F}, \theta_t A = A$ for all $t \in \mathbb{R}$ implies $\mathbb{P}(A) \in \left\{0,1\right\}$. The second component is a mapping that models the dynamics of the system. This is a $(\mathscr{B}(\mathbb{R}) \otimes \mathscr{F} \otimes \mathscr{B}(X), \mathscr{B}(X))$-measurable function $\phi:\mathbb{R} \times \Omega \times X \to X$ such that \footnote{Here and throughout the paper we will use both the equivalent notations $\phi(t,\omega,x)$ and $\phi(t,\omega)x$.}
\begin{enumerate}[label=(\roman*)]
	\setlength{\itemsep}{0pt}
	\setlength{\parskip}{0pt}
	\item $\phi(0,\omega,x) =x$ for all $x \in X$,
	\item $\phi(t+s,\omega,x) = \phi(t,\theta_s \omega, \phi(s,\omega,x))$ for all $t,s \in \mathbb{R}$ and $x \in X$, $\mathbb{P}$-almost surely 
	(\emph{cocycle property}).\footnote{This definition of random cocycle follows the convention in e.g. \cite{bates2009random,fehrman2019well}. In Arnold \cite{arnold1995random}, the cocycle property is required to hold for \emph{all} $\omega \in \Omega$, instead of almost surely. In case the cocycle exists for almost all $\omega \in \Omega$ only, $\phi$ is called a \emph{crude cocycle} and through a \emph{perfection procedure} it possible to define an indistinguishable RDS for which the cocycle property is fulfilled for all noise realizations, see e.g. \cite[Chapter 4.10]{duan2014effective}  and references therein, most notably \cite{flandoli2004stationary,arnold1995perfect}.}
\end{enumerate}
The skew-product structure $\Theta:\mathbb{R} \times \Omega \times X \to \Omega \times X$ characterizing the random dynamical system $(\theta,\phi)$ can be succinctly written as 
\begin{equation*}
\Theta(t)(\omega,x) := (\theta_t \omega, \phi(t,\omega,x)).
\end{equation*}
A probability measure $\mu$ on $(\Omega \times X, \mathscr{F} \otimes \mathscr{B}(X))$ is said to be \emph{invariant} if
\begin{enumerate}[label=(\roman*)]
	\setlength{\itemsep}{0pt}
	\setlength{\parskip}{0pt}
	\item $\mu(\Theta_t A) = \mu(A)$ for all $t \in \mathbb{R}$ and $A \in \mathscr{F} \otimes \mathscr{B}(X)$
	\item The marginal of $\mu$ on $(\Omega,\mathscr{F})$ is $\mathbb{P}$.
\end{enumerate}
The canonical construction of the sample path space of Brownian motions can be briefly outlined as follows. Let $\Omega:= C_0(\mathbb{R},\mathbb{R}^{2d})$ be the space of all continuous functions $\xi:\mathbb{R} \to \mathbb{R}^{2d}$ such that $\xi(0) = 0$, endowed with the compact open topology. Let $\mathscr{F} = \mathscr{B}(\Omega)$ denote the Borel $\sigma$-algebra on $\Omega$. Then, there exists the so-called \emph{Wiener probability measure} $\mathbb{P}$ on $(\Omega,\mathscr{F})$ ensuring that the processes $(B(t))_{t \in \mathbb{R}}$ and $(W(t))_{t \in \mathbb{R}}$
are independent $d$-dimensional Brownian motions, with corresponding sample paths $(\omega, \beta) :=\xi \in \Omega= \Omega_W \times \Omega_B$ where $\Omega_W$ and $\Omega_B$ denote the intrinsic and common noise sample spaces. 
The natural filtration is the $\sigma$-algebra $\mathscr{F}_{s,t}$ generated by $\xi(u)-\xi(v)$ for $s \leq v \leq u \leq t$. The Wiener measure $\mathbb{P}$ is ergodic with respect to the Wiener shift map $\theta_t : \Omega \to \Omega$ defined by
\begin{equation*}
(\theta_t \xi)(s) := \xi(s+t) - \xi(t), \ \ s \in \mathbb{R}.
\end{equation*}
Therefore, $(\Omega,\mathscr{F},\mathbb{P},(\theta_t)_{t \in \mathbb{R}})$ is an ergodic random dynamical system. 

With this sample path evolution as explicit representation of the noise, the SDE~\eqref{SDE} is a random dynamical system on the product $ \Omega\times \mathbb{R}^d$.  We find that the stochastic Fokker-Planck equation \eqref{stochasticFP} is a random dynamical system on $\Omega_B\times\mathcal{S}$, where $\mathcal{S}$ is the Schwartz space of solutions of \eqref{stochasticFP} identified in Theorem~\ref{thm:existence_uniqueness_regularity}.

%--------------------------------------- Prop: SFPE is a RDS  -----------------------------------------%
\begin{proposition}\label{prop:SFPE_RDS}
The stochastic Fokker-Planck equation~\eqref{stochasticFP} is a random dynamical system.
\end{proposition}

Next, we show that \eqref{stochasticFP} possesses a unique global pullback attractor in the Schwartz space $\mathcal{S}$ of rapidly decaying functions for $\mathbb{P}_B$-almost all $\beta \in \Omega_B$. 
%--------------------------------------- Theorem: Pullback attractor -----------------------------------%
\begin{theorem}[Pullback attractor]\label{thm:pullback_attractor}
Let $\Phi$ be the random dynamical system associated to \eqref{stochasticFP}. Then, for $\mathbb{P}_B$-almost all $\beta \in \Omega_B$ \eqref{stochasticFP} has a unique pullback attractor defined by
\begin{equation*}
p_\beta := \lim_{\tau \to \infty} \Phi(\tau,\theta_{-\tau}\beta)p
\end{equation*}
which is independent of $p\in L^1$. Moreover, $p_\beta \in \mathcal{S}$ and convergence is with respect to the semi-norm \eqref{seminorm} on $\mathcal{S}$.
\end{theorem}

By the \emph{Correspondence Theorem} (see e.g. Arnold~\cite[Remark 1.4.2 and Proposition 1.4.3]{arnold1995random}  and Crauel and Flandoli~\cite[Section 4]{crauel1994attractors}), if \eqref{SDE} has a unique stationary measure, then the corresponding random dynamical system has a unique invariant \emph{Markov} measure, i.e. an invariant measure that is measurable with respect to the past,\footnote{See Proposition~\ref{prop:disintegration}(i).} and pullback attractors are identified with disintegrations of this Markov measure. We show in Proposition~\ref{prop:cnpbversusdisintegration} that the common noise pullback attractors of \eqref{stochasticFP} are equal to the expectation of the pullback attractors of \eqref{SDE} with respect to %realisations of 
the intrinsic noise $W$, for a single fixed common noise realisation $\beta$. We summarize some well-established results on Markov measures and their disintegration \cite{arnold1995random} in the context of our setting. 

%---------------------------- Prop: Markov measure and its disintegration -----------------------------%
\begin{proposition}[Markov measure and its disintegration]\label{prop:disintegration}
Let $\rho$ be the (unique) stationary measure of the random dynamical system $\phi$ associated to \eqref{SDE} and
\begin{equation}\label{pb}
\mu_{\omega,\beta} := \lim_{\tau \to \infty} \phi(\tau,\theta_{-\tau} \omega, \theta_{-\tau} \beta)_{*} \rho.
\end{equation}
Then $\left\{\mu_{\omega,\beta} \right\}_{(\omega,\beta) \in \Omega}$ is $\mathbb{P}$-a.e. unique on $\mathscr{B}(\mathbb{R}^d)$ and
\begin{enumerate}[label=(\roman*)]
\item for all $C \in \mathscr{B}(\mathbb{R}^d), (\omega,\beta) \to \mu_{\omega,\beta}(C)$ is $ \mathscr{F}^{-}$-measurable, where $\mathscr{F}^{-}=\sigma\left(\cup_{s\leq} \mathscr{F}_{s,t}\right)$.
\item for $\mathbb{P}$-a.e. $(\omega, \beta) \in \Omega_W \times \Omega_B$, $\mu_{\omega,\beta}$ is a probability measure on $(X,\mathscr{B}(\mathbb{R}^d))$.
\item for all $A \in \mathscr{F} \otimes \mathscr{B}(\mathbb{R}^d)$
\begin{equation*}
\begin{split}
\mu(A) & := \int_{\Omega_B} \int_{\Omega_W} \int_X \mathbbm{1}_A(\omega,\beta,x) \mu_{\omega,\beta}(dx) \mathbb{P}_W(d \omega) \mathbb{P}_B(d \beta) \\
& = \int_{\Omega_B} \int_{\Omega_W} \mu_{\omega,\beta}(A_{\omega,\beta}) \mathbb{P}_W(d \omega) \mathbb{P}_B(d \beta),
\end{split}
\end{equation*}
where
\begin{equation*}
A_{\omega,\beta} := \left\{x: (\omega,\beta,x) \in A \right\},
\end{equation*}
is an invariant probability measure of $\phi$. The measure $\mu$ is known as the \emph{Markov measure} associated to $\rho$ and it is the unique invariant probability measure of $\phi$ that is measurable with respect to the past, cf.~(i), such that
\[
\int_{\Omega_B} \int_{\Omega_W} \mu_{\omega,\beta}\mathbb{P}_W(d\omega)\mathbb{P}_B(d\beta)=\rho.
\]
\end{enumerate}
\end{proposition}

Against this background, we now prove that the common noise pullback attractor of \eqref{stochasticFP} is the expectation of the pullback attractor of the underlying SDE \eqref{SDE} with respect to the intrinsic noise.

%-------------------------------- Prop: Common noise pullback attractor ------------------------------------%
\begin{proposition}[Common noise pullback attractor]\label{prop:cnpbversusdisintegration}
Let $\Phi$ be the random dynamical system associated to \eqref{stochasticFP}, $p_\rho$ be the Lebesgue density of the stationary measure $\rho$ of \eqref{SDE}, and
\begin{equation*}
\mu_\beta:= \int_{\Omega_W} \mu_{\omega,\beta} \mathbb{P}_W(d \omega).
\end{equation*}
Then, $\mathbb{P}_B$-a.s., $\mu_\beta$ is a probability measure on $\mathscr{B}(\mathbb{R}^d)$ with Lebesgue density $p_\beta$, where 
\begin{equation*}
p_\beta = \lim_{\tau \to \infty} \Phi(\tau,\theta_{-\tau}\beta) p_\rho
\end{equation*}
and 
\begin{equation*}
\int_{\Omega_B} p_\beta\mathbb{P}_B(d\beta)=p_\rho.
\end{equation*}
\end{proposition}

We finally stipulate a direct consequence of ergodicity for observables $g:\mathcal{S}\to\mathbb{R}$.

%----------------------------------------- Prop: Ergodicity ------------------------------------------------%
\begin{proposition}\label{prop:erg}
Let $g:\mathcal{S}\to\mathbb{R}$ be continuous and integrable, then
\begin{equation}\label{eq:ergo_result}
\lim_{\tau\to\infty}\frac{1}{\tau}\int_0^\tau g(\Phi(\tau,\beta) p) dt=\int_{\Omega_B} g(p_\beta)\mathbb{P}_B(d\beta),
\end{equation}
$\mathbb{P}_B$-almost surely, independent of $p\in \mathcal{S}$.
\end{proposition}

%%%%%%%%%%%%%%%%%%%%%%%%%%%%%%%%%%%%%%%%%%%%%%%%%%%%%%%%%%%%%%%%%%%%%%%%%%%%%%%%%%%%%%%%%%%%%%%%%%%%%%%%%%%%%%
%											Appendix
%%%%%%%%%%%%%%%%%%%%%%%%%%%%%%%%%%%%%%%%%%%%%%%%%%%%%%%%%%%%%%%%%%%%%%%%%%%%%%%%%%%%%%%%%%%%%%%%%%%%%%%%%%%%%%
\section*{Appendices}
\addcontentsline{toc}{section}{Appendices}
\appendix
%\section{Appendices}\label{sec:appendix}
%-------------------------------------------------------------------------------------------------------------
%									Proofs of the results from Section 2
%-------------------------------------------------------------------------------------------------------------
\section{Proof of Theorem~\ref{thm:existence_uniqueness_regularity}}\label{app:non_auto_FP_IVP_proof}

We start this section by demonstrating the equivalence between total mass conservation for measures and a weighted integrability condition to be fulfilled by a weak solution $q(t)$ of the initial value problem~\eqref{eq:non autonomous FP gradient}. As we shall discuss below, this condition will be employed to prove an important $L^1$ estimate (Lemma~\ref{lem:L1_estimate}) which, in turn, will be crucial for establishing uniqueness of solutions (Lemma~\ref{lem:uniqueness}).

%-------------------------------------- Lemma: Probability solution  ------------------------------------------%
\begin{lemma}[Equivalence between mass conservation and the weighted integrability condition]\label{lem:probability_solution}
A weak solution $q$ of the non-autonomous Fokker-Planck equation~\eqref{eq:non autonomous FP gradient} with initial condition $q_0 \in \mathcal{P}(\mathbb{R}^d)$ is a probability solution for any given $\beta \in C^{1/2}(\mathbb{R},\mathbb{R}^d)$, i.e.
\begin{equation*}
q(t,x;\beta) \geq 0, \ \ \ \int_{\mathbb{R}^d} q(t,x;\beta) dx = 1\ \text{for all}\ t \geq 0,
\end{equation*}
if and only if the \emph{weighted integrability condition}
\begin{equation}\label{weighted_integr_condition}\tag{W.I.C.}
\lim_{N \to \infty} \int_0^T \int_{N<\|x\|<2 N} \|x\|^{-1} \|\nabla U\| \|q(t,x;\beta)\| dx dt = 0
\end{equation}
holds.
\begin{proof}
Let us define the test function $\varphi_N:= \vartheta(\frac{x}{N})$ for any $N \in \mathbb{N}$, where $\vartheta \in C_c^\infty(\mathbb{R}^d,\mathbb{R}_+)$ is a cut-off function such that $\vartheta' \leq 0$ and
\begin{equation*}
\vartheta(z) = 
\begin{cases}
1 & \text{if}\ z \in [0,1]^d \\
2-z & \text{if}\ z \in [1,2]^d.
\end{cases}
\end{equation*}
$\vartheta$ is constructed to be at least $C^2$ at $z=1$ and $z=2$. It is assumed to be equal to $0$ for $z \in [2,\infty)^d$ and extended evenly for $z \in (-\infty,0]^d$. Testing the non-autonomous Fokker Planck equation~\eqref{eq:non autonomous FP gradient} with $\varphi_N$, integrating by parts and rearranging terms we obtain
\begin{equation*}
\langle q(0), \varphi_N \rangle - \int_0^T \int_{\mathbb{R}^d} \nabla U \nabla \varphi_N q dx dt = \langle q(T), \varphi_N \rangle - \int_0^T \langle q, \Delta \varphi_N \rangle dt,
\end{equation*}
where we suppressed the dependence on $(x,\beta)$ in order to simplify the notation. Since in the limit $N \to \infty$ $\Delta \varphi_N \sim N^{-2}$ and
\begin{equation*}
\begin{split}
\lim_{N \to \infty} \langle q(T), \varphi_N \rangle & = \|q(T)\|_{L^1} \\
\lim_{N \to \infty} \langle q(0), \varphi_N \rangle & = \|q(0)\|_{L^1},
\end{split}
\end{equation*}
we immediately obtain
\begin{equation*}
\lim_{N \to \infty} \int_0^T \int_{N<\|x\|<2 N} \|x\|^{-1} \nabla U q dx dt = \bigg| \|q(T)\|_{L^1} - \|q(0)\|_{L^1} \bigg|
\end{equation*}
and therefore we conclude.
\end{proof}
\end{lemma}

%------------------------  Remark: Sharpness of the weighted integrability condition   ------------------------%
\begin{remark}[Probability and bounded measures]
The problems of uniqueness in the class of probability measures and in the class of all bounded measures are not equivalent. Consider for simplicity the one dimensional case, the autonomous scenario $\beta \equiv 0$ and the potential
\begin{equation*}
V'(x) = 4 x^3 + 16 x^3 (1 + 4 x^4)^{-1}.
\end{equation*}
Then, the Fokker-Planck equation~\eqref{eq:non autonomous FP gradient} admits the following stationary, bounded sign-changing solution
\begin{equation*}
q(x) = x (1 + 4 x^4)^{-1}.
\end{equation*}
The weighted integrability condition is violated, since
\begin{equation*}
\int_{\mathbb{R}} x^2 |q(x)| dx = +\infty.
\end{equation*}
At the same time, \eqref{eq:non autonomous FP gradient} admits the well known stationary probability solution
\begin{equation*}
\rho(x) = e^{-V(x)} = (1+4 x^4)^{-1} e^{-x^4}.
\end{equation*}
Therefore, there exists a unique solution in the class of probability measures, but there are also nonzero signed solutions in the class of bounded measures. For further details, see \cite{bogachev2015fokker}, Example 4.1.3.
\end{remark}
%-------------------------------------------------------------------------------------------------------------%

Theorem~\ref{thm:existence_uniqueness_regularity} is proved by combining a series of energy-type estimates. We remark that, although the focus of this theorem is on probability measures, from here onwards we consider more broadly the evolution of signed measures. This is needed for the proof of uniqueness of the weak solution in Lemma~\ref{lem:uniqueness}, where the evolution of the difference between two probability solutions is considered. 

We make the following key assumptions:
\begin{assumption} \label{assumptions_estimates} 
\noindent
\begin{enumerate}[label=(\Roman*)]
\item Weak $L^1$ solutions of the Fokker Planck equation~\eqref{eq:non autonomous FP gradient} are required to satisfy the weighted integrability condition~\eqref{weighted_integr_condition}.
\item The initial condition $q_0$ is a signed measure.
\item The potential $U$ is infinitely differentiable and satisfies the dissipation condition~\eqref{eq:dissipation_assumption}.
\end{enumerate}
\end{assumption}
Restricted to the setting of probability measures, Assumption~\ref{assumptions_estimates} boils down to the setting of Theorem~\ref{thm:existence_uniqueness_regularity}. 
%For the sake of clarity, we state them below.

Next, we note that the space $L^1(\mathbb{R}^d)$ can be interpreted as regular measures and embedded isometrically into the space of signed Borel measures $\mathcal{M}(\mathbb{R}^d)$. Although this result is already known, for the sake of having a self-contained discussion, we provide below an explicit proof, adapting the one in \cite{pdes23measures}, Proposition 2.7.

%-------------------------------------- Lemma: Approximation Borel measures  ---------------------------------%
\begin{lemma}[Approximation of signed measures]\label{lem:approximation_borel}
	For every measure $\mu \in \mathcal{M}(\mathbb{R}^d)$, there exists a sequence of (signed) measures $(\mu_n)_{n \in \mathbb{N}} \in L^1(\mathbb{R}^d)$ such that
	\begin{equation*}
	\| \mu_n \|_{L^1} \leq \| \mu \|_{\mathcal{M}(\mathbb{R}^d)}
	\end{equation*}
	for any $n \in \mathbb{N}$ and
	\begin{equation*}
	\lim_{n \to \infty} \langle \mu_n, \varphi \rangle = \langle \mu, \varphi \rangle
	\end{equation*}
	for all test functions $\varphi \in C_0(\mathbb{R}^d)$, where $C_0(\mathbb{R}^d)$ denotes the space of all continuous functions with compact support on $\mathbb{R}^d$.
	\begin{proof}
		Let $(\varrho_n)_{n \in \mathbb{N}}$ be a sequence of mollifiers, that is, for every $n \in \mathbb{N}, \varrho_n \in C_0^{\infty}(\mathbb{R}^d)$, $\varrho_n$ is nonnegative, such that
		\begin{equation*}
		\int_{\mathbb{R}^d} \varrho_n = 1
		\end{equation*}
		and for every $\delta>0$
		\begin{equation*}
		\lim_{n \to \infty} \int_{\mathbb{R}^d \setminus B(0,\delta)} \varrho_n = 0,
		\end{equation*}
		where $B(0,\delta)$ denotes the open ball centred at $0$ with radius $\delta$. Then, we consider the convolution $\mu_n:=\varrho_n \ast \mu$ between the mollifier and the measure $\mu \in \mathcal{M}(\mathbb{R}^d)$. We immediately have $\mu_n \in L^1(\mathbb{R}^d)$. From the convolution definition and thanks to Fubini's theorem, we deduce
		\begin{equation*}
		\begin{split}
		\int_{\mathbb{R}^d} \varphi d(\varrho_n \ast \mu) & = \int_{\mathbb{R}^d} \varphi(x) \int_{\mathbb{R}^d} \varrho_n(x-y) d \mu(y) dx \\
		&  = \int_{\mathbb{R}^d} \int_{\mathbb{R}^d}\varrho_n(x-y) \varphi(x) dx d \mu(y) \\
		& = \int_{\mathbb{R}^d} \varrho_n \ast \varphi d \mu.
		\end{split}
		\end{equation*}
		Since by construction $\varphi \in C_0(\mathbb{R}^d)$, the sequence $(\varrho_n \ast \varphi)_{n \in \mathbb{N}}$ converges uniformly to $\varphi$ on $\mathbb{R}^d$. Hence, $\mu_n \to \mu$ weakly. Finally, we observe that
		\begin{equation*}
		|\mu_n| \leq \int_{\mathbb{R}^d} \varrho_n(x-y) d |\mu|(y).
		\end{equation*}
		Applying again Fubini's theorem we obtain
		\begin{equation*}
		\|\mu_n\|_{L^1} \leq \int_{\mathbb{R}^d} \left(\int_{\mathbb{R}^d} \varrho_n(x-y) dx \right) d |\mu|(y) \leq \int_{\mathbb{R}^d} d |\mu| (y) = \|\mu\|_{\mathcal{M}(\mathbb{R}^d)}
		\end{equation*}
		and we conclude.
	\end{proof}
\end{lemma}
Since our result will imply $q(t) \in L^1(\mathbb{R}^d)$ for all $t>0$, thanks to Lemma~\eqref{lem:approximation_borel}, in the proof of Theorem~\ref{thm:existence_uniqueness_regularity} we can restrict ourselves to the case $q_0 \in L^1(\mathbb{R}^d)$.

Our strategy proceeds as follows. First, we prove the $L^1$ estimate~\eqref{L1_estimate} by employing the weighted integrability condition~\eqref{weighted_integr_condition}. Thanks to~\eqref{L1_estimate}, we establish uniqueness in Lemma~\ref{lem:uniqueness}. Next, we prove the weighted $L^1$-estimate~\eqref{weighted_L1_estimate}, which controls the behaviour of the tails of the solution. Following the same ideas, we prove the $L^1$-localization estimate~\eqref{L1_localization} and the two smoothing estimates~\eqref{first_smoothing_estimate} and \eqref{second_smoothing_estimate}. These results establish that the unique solution of the Fokker Planck equation~\eqref{eq:non autonomous FP gradient} is smooth in the space variable $x$, as regular in time as the function $\beta$ and belongs to $L^2(\mathbb{R}^d)$ at all times $t>0$. Exploiting the structure of the equation and the infinite differentiability of the potential, we iterate the argument and achieve infinite differentiability and rapidly decreasing behaviour of all derivatives.

The $L^1$-estimate~\eqref{L1_estimate} will be proved with the help of inequality~\eqref{ineq:weak solution} below. Given a weak solution $q$ of the non-autonomous Fokker Planck equation~\eqref{eq:non autonomous FP gradient}, the fundamental idea to prove this inequality consists in regularizing the equation with a mollifier $\vartheta_\delta$ first and then taking the limit $\delta \to 0$. Let $\vartheta \in C_0^{\infty}(\mathbb{R}^d, \mathbb{R}_+)$ denote a non-negative mollification kernel satisfying $\int_{\mathbb{R}^d} \vartheta(x) dx = 1$ and define the standard Dirac's delta approximation
\begin{equation*}
\vartheta_\delta(x) := \frac{1}{\delta} \vartheta\left(\frac{x}{\delta}\right).
\end{equation*}
Observe that, since $\vartheta$ has a compact support, 
\begin{equation}
\nabla_x \vartheta_\delta(x-y) \ne 0 \ \ \text{if}\ |x-y| \leq C \delta
\label{eq:property_theta_delta}
\end{equation}
for some positive constant $C$. Moreover, using integration by parts, 
\begin{equation}\label{eq:property_integral_theta_prime}
\int_{\mathbb{R}^d} \vartheta'(s) s ds = - \int_{\mathbb{R}^d} \vartheta(s) ds = -1.
\end{equation}
We denote by $S_\delta$ the mollification operator
\begin{equation}
(S_\delta(q))(x) := \int_{\mathbb{R}^d} \vartheta_\delta(x-y) q(y) dy
\label{eq:mollification_operator}
\end{equation}
and by $\text{sgn}_\gamma(x)$ the standard smooth and monotone approximation of $\text{sgn}(x)$, i.e.
\begin{equation}\label{eq:smooth_approx_sign}
\text{sgn}_\gamma(x) := \frac{x}{\sqrt{x^2 + \gamma^2}}.
\end{equation}
Finally, we define
\begin{equation*}
|z|_\gamma := \int_0^z \text{sgn}_\gamma(s) ds.
\end{equation*}
and notice that, by construction, $\lim_{\gamma \to 0} |z|_\gamma = |z|$.

%---------------------------------------- Lemma: Weak solution inequality  ------------------------------------%
\begin{lemma}[Weak solution inequality]
Any weak solution of \eqref{eq:non autonomous FP gradient} satisfies the inequality
\begin{equation}\label{ineq:weak solution}
\frac{d}{d t} \langle \|q(\cdot,t;\beta)\|, \varphi \rangle \leq \langle \partial_t \varphi - \nabla U \nabla \varphi + \Delta \varphi, \|q(\cdot,t;\beta)\| \rangle
	\end{equation}
	for almost all $t \geq 0$, any $\beta \in C^{1/2}(\mathbb{R},\mathbb{R}^d)$ and $\varphi \in C_0^{\infty}((0,T) \times \mathbb{R}^d)$ such that $\varphi \geq 0$.
	\begin{proof}
		We apply the mollification operator $S_\delta$ defined in~\eqref{eq:mollification_operator} to both sides of the non-autonomous Fokker-Planck equation \eqref{eq:non autonomous FP gradient}, thereby obtaining
		\begin{equation*}
		\partial_t (S_\delta(q)) = \nabla (S_\delta(\nabla U q)) + \Delta (S_\delta(q)), \ \ S_\delta(q)_{\big |_{t=0}} = S_\delta(q_0).
		\end{equation*}
		We define the test function
		\begin{equation*}
		\psi(t,x) := \varphi(t,x) \text{sgn}_\gamma (S_\delta(q)(x)),
		\end{equation*}
		where $\varphi$ is the test function in~\eqref{ineq:weak solution} and $\text{sgn}_\gamma$ is the smooth approximation of the $\text{sgn}$ function defined in~\eqref{eq:smooth_approx_sign}. We obtain
		\begin{equation}\label{eq:fun_weak_sol_1}
		\frac{d}{d t} \langle |S_\delta(q)|_\gamma, \psi \rangle = \underbrace{\langle \partial_t \psi, |S_\delta(q)|_\gamma \rangle}_{a} +  \underbrace{\langle \nabla (S_\delta (\nabla U q)), \psi \rangle}_{b} + \underbrace{\langle \Delta S_\delta(q), \psi \rangle}_{c},
		\end{equation}
		The idea now is to write each term on the RHS of~\eqref{eq:fun_weak_sol_1} in a convenient form by means of integration by parts and take the limit $\delta \to 0$. Firstly, integration by parts implies term $(a)$ can be rewritten as
		\begin{equation}\label{eq:fun_weak_sol_2}
		\langle \Delta S_\delta(q), \psi \rangle = - \langle \nabla S_\delta(q), \text{sgn}^{'}_\gamma (S_\delta(q)) \nabla S_\delta(q) \varphi \rangle - \langle \nabla S_\delta(q), \text{sgn}_\gamma (S_\delta(q)) \nabla \varphi \rangle,
		\end{equation}
		where ${'}$ denotes the derivative, since boundary terms vanish thanks to the dissipation condition~\eqref{eq:dissipation_assumption}. Dropping the first term on the RHS of~\eqref{eq:fun_weak_sol_2} and applying integration by parts again yields
		\begin{equation*}
		\langle S_\delta(q), \psi \rangle \leq \langle \Delta \varphi, |S_\delta(q)|_\gamma \rangle.
		\end{equation*}
		Taking the limit,
		\begin{equation*}
		\lim_{(\gamma,\delta) \to (0,0)} \langle \Delta \varphi, |S_\delta(q)|_\gamma \rangle = \langle |q|, \Delta \varphi \rangle.
		\end{equation*}
		Next, term $(b)$ in~\eqref{eq:fun_weak_sol_1} reads as
		\begin{subequations}
			\begin{align}
			\langle \nabla (S_\delta (\nabla U q)), \varphi\ \text{sgn}_\gamma(S_\delta(q)) \rangle & = \langle \nabla (\nabla U S_\delta(q)), \varphi\ \text{sgn}_\gamma (S_\delta(q)) \rangle \label{eq:fun_weak_sol_4_a} \\
			& \quad + \langle \nabla (S_\delta (\nabla U q) - \nabla U S_\delta(q)), \varphi\ \text{sgn}_\gamma (S_\delta(q)) \rangle. \label{eq:fun_weak_sol_4_b}
			\end{align}
		\end{subequations}
		The first term~\eqref{eq:fun_weak_sol_4_a} can be written as
		\begin{equation*}
		\begin{split}
		\langle \nabla (\nabla U S_\delta(q)), \varphi\ \text{sgn}_\gamma (S_\delta(q)) \rangle & = \langle \Delta U \varphi, S_\delta(q)\ \text{sgn}_\gamma (S_\delta(q)) \rangle  + \langle \nabla U \psi, \nabla |S_\delta(q)|_\gamma \rangle \\
		& = \langle \Delta U \varphi, S_\delta(q)\ \text{sgn}_\gamma (S_\delta(q)) \rangle  - \langle \nabla (\nabla U \varphi), |S_\delta(q)|_\gamma \rangle \\
		& \to \langle \Delta U \varphi, |q| \rangle - \langle \nabla (\nabla U \varphi), |q| \rangle  \\
		& = - \langle \nabla U \nabla \varphi, |q| \rangle
		\end{split}
		\end{equation*}
		as $(\gamma,\delta) \to (0,0)$. In order to write more explicitly the term~\eqref{eq:fun_weak_sol_4_b}, we recall that
		\begin{equation*}
		S_\delta(\nabla U q)(x) - \nabla U S_\delta(q)(x) = \int_{\mathbb{R}^d} \vartheta_\delta(x-y) \left[\nabla_y U - \nabla_x U \right] q(y) dy
		\end{equation*}
		and therefore
		\begin{subequations}
			\begin{align}
			& \langle \nabla (S_\delta(\nabla U q) - \nabla U S_\delta(q)), \varphi\ \text{sgn}_\gamma (S_\delta(q)) \rangle \nonumber \\
			& = \int_{\mathbb{R}^d} \int_{\mathbb{R}^d} \nabla \vartheta_\delta(x-y) \left(\nabla_y U - \nabla_x U \right) q(y) \varphi(x) \text{sgn}_\gamma (S_\delta(q)(x)) dy dx \label{eq:fun_weak_sol_5_a} \\
			& \quad - \int_{\mathbb{R}^d} \int_{\mathbb{R}^d} \vartheta_\delta(x-y) \Delta U(x) q(y) \varphi(x) \text{sgn}_\gamma (S_\delta(q)(x)) dy dx \label{eq:fun_weak_sol_5_b}.
			\end{align}
		\end{subequations}
		Taking the limit $\delta \to 0$ in~\eqref{eq:fun_weak_sol_5_b} yields
		\begin{equation*}
		\lim_{\delta \to 0} \int_{\mathbb{R}^d} \int_{\mathbb{R}^d} \vartheta_\delta(x-y) \Delta U(x) q(y) \varphi(x) \text{sgn}_\gamma (S_\delta (q)(x)) dy dx = \langle \Delta U \varphi, \text{sgn}_\gamma(q) \rangle.
		\end{equation*}
		For what concerns the integral term~\eqref{eq:fun_weak_sol_5_a}, instead, we use~\eqref{eq:property_theta_delta}. Consequently, we may write
		\begin{equation*}
		\nabla_y U - \nabla_x U = - \Delta U(x) (x-y) + O(\delta^2(\|x\|+1)).
		\end{equation*}
		Thus, at first order approximation, \eqref{eq:fun_weak_sol_5_a} reduces to
		\begin{equation*}
		- \int_{\mathbb{R}^d} \int_{\mathbb{R}^d} \nabla \vartheta_\delta(x-y) \Delta U(x) (x-y) q(y) \varphi(x) \text{sgn}_\gamma (S_\delta(q)(x)) dy dx.
		\end{equation*}
		Thanks to~\eqref{eq:property_integral_theta_prime}, we deduce
		\begin{equation*}
		\begin{split}
		& \lim_{\delta \to 0} - \int_{\mathbb{R}^d} \int_{\mathbb{R}^d} \nabla \vartheta_\delta(x-y) \Delta U(x) (x-y) q(y) \varphi(x) \text{sgn}_\gamma (S_\delta(q)(x)) dy dx \\
		& = \langle \Delta U \varphi, \text{sgn}_\gamma(q) \rangle.
		\end{split}
		\end{equation*}
		Hence, \eqref{eq:fun_weak_sol_4_b} vanishes as $\delta \to 0$. Putting everything together, we let $(\gamma,\delta) \to (0,0)$, observe that $\psi \to \varphi$ and $|S_\delta(q)|_\gamma \to |q|$ and finally obtain the inequality \eqref{ineq:weak solution}.
	\end{proof}
\end{lemma}

We are now ready to prove the following $L^1$-estimate, the proof of which heavily relies on the weighted integrability condition~\eqref{weighted_integr_condition}.

%------------------------------------------ Lemma: L^1-estimate ----------------------------------------------%
\begin{lemma}[$L^1$-estimate]\label{lem:L1_estimate}
	Any weak solution of the non-autonomous Fokker Planck equation~\eqref{eq:non autonomous FP gradient} satisfies for any given $\beta \in C^{1/2}(\mathbb{R},\mathbb{R}^d)$ the $L^1$-estimate
	\begin{equation}
	\|q(\cdot,t;\beta) \|_{L^1} \leq \|q(\cdot,0;\beta) \|_{L^1}, \ \ \ t>0.
	\label{L1_estimate}
	\end{equation}
	\begin{proof}
		Let us consider a cut-off function $\vartheta \in C_0^{\infty}(\mathbb{R}^d,[0,1])$ such that
		\begin{equation*}
		\vartheta(z) = 
		\begin{cases}
		1 & \text{if}\ z \in [-1,1]^d \\
		0 & \text{if}\ z \in [2,\infty)^d 
		\end{cases}
		\end{equation*}
		and the test function $\varphi_N := \vartheta \left(\frac{x}{N}\right)$ for some $N \in \mathbb{N}$. Then, using the fact that the function $t \to \langle |q(t)|, \varphi_N \rangle$ is continuous in time, inequality~\eqref{ineq:weak solution} implies
		\begin{equation}\label{eq:fun_weak_sol_3}
		\langle |q(T)|, \varphi_N \rangle \leq \langle |q(\tau)|, \varphi_N \rangle + \int_\tau^T \langle \Delta \varphi_N - \nabla U \nabla \varphi_N, \|q\| \rangle dt,
		\end{equation}
		where we suppressed the dependence on $x$ and $\beta$ in order to simplify the notation. The weighted integrability condition ~\eqref{weighted_integr_condition} ensures that the integral on the RHS of~\eqref{eq:fun_weak_sol_3} tends to $0$ as $N \to +\infty$. Indeed, 
		\begin{equation*}
		\begin{split}
		\big | \langle \nabla U \nabla \varphi_N, \|q\| \rangle \big | & \leq \frac{1}{N} \int_{N < \|x\| < 2 N} \|\nabla U\| \bigg|\vartheta' \left(\frac{x}{N} \right) \bigg| \|q\| dx \\
		& \leq C \int_{N < \|x\| < 2 N} \|x\|^{-1} \|\nabla U\| \|q\| dx 
		\end{split}
		\end{equation*}
		and
		\begin{equation*}
		\int_\tau^T \langle \nabla U \nabla \varphi_N, \|q\| \rangle dt \leq \int_\tau^T \int_{N < \|x\| < 2 N} \|x\|^{-1} \|\nabla U\| \|q\| dx dt \to 0,
		\end{equation*}
		as $N \to \infty$. Passing to this limit in~\eqref{eq:fun_weak_sol_3} therefore yields
		\begin{equation*}
		\| q(\tau)\|_{L^1} \leq \|q(0)\|_{L^1}.
		\end{equation*}
		Finally, we let $\tau \to 0$, exploit continuity and conclude.
	\end{proof}
\end{lemma}
 
Uniqueness of the weak solution for the initial value problem~\eqref{eq:non autonomous FP gradient} is an immediate consequence of Lemma~\ref{lem:L1_estimate}.

%-------------------------------------------- Lemma: Uniqueness ----------------------------------------------%
\begin{lemma}[Uniqueness]\label{lem:uniqueness}
The initial value problem ~\eqref{eq:non autonomous FP gradient} admits a unique weak solution for any given $\beta \in C^{1/2}(\mathbb{R},\mathbb{R}^d)$.
	\begin{proof}
		Let us denote by $q_1,q_2$ two distinct weak solutions of~\eqref{eq:non autonomous FP gradient} with the same initial condition $q_0$. Define $q(x,t;\beta) := q_1(x,t;\beta) - q_2(x,t;\beta)$. Then, thanks to linearity, $q$ will be a weak solution of~\eqref{eq:non autonomous FP gradient} and
		\begin{equation*}
		\|q(\cdot,t;\beta)\|_{L^1} \leq \|q(\cdot,0;\beta)\|_{L^1} = 0
		\end{equation*}
		thanks to the $L^1$-estimate~\eqref{L1_estimate}.
	\end{proof}
\end{lemma}
Next, we proceed with proving the following weighted $L^1$-estimate, which describes the global behaviour of the tails of the weak solution.

%------------------------------------- Lemma: Weighted L^1-estimate ------------------------------------------%
\begin{lemma}[Weighted $L^1$-estimate]\label{lem:weighted_L1_estimate}
The unique weak solution of the non-autonomous Fokker Planck equation~\eqref{eq:non autonomous FP gradient} satisfies for any given $\beta \in C^{1/2}(\mathbb{R},\mathbb{R}^d)$ the weighted $L^1$-estimate
\begin{equation}
\| (1+x^n) q(\cdot,t;\beta) \|_{L^1} \leq C(\beta) \|(1+x^n) q(\cdot,0;\beta) \|_{L^1}
\label{weighted_L1_estimate}
\end{equation}
for any $n\ \in \mathbb{N}$ and some constant $C=C(\beta)<\infty$. \footnote{In this and subsequent lemmas time-dependence of the constant $C$ is not problem. Since we are interested in local regularity for $t>0$, we are considering $t \in [0,T]$. Without loss of generality, we might set $T=1$. For what concerns the $\beta$-dependence instead, the constant $C$ in general will not be uniform with respect to $\beta$. To gain uniformity, additional assumptions on $\beta$ would be required (such as boundedness), but in our context this is not needed.}

\begin{proof}
	We multiply the non-autonomous Fokker Planck equation~\eqref{eq:non autonomous FP gradient} by $(1+x^n) \text{sgn}(q)$ for any $n \in \mathbb{N}$ and integrate over $\mathbb{R}^d$:
	\begin{equation*}
	\begin{split}
	\frac{d}{dt} \| (1+x^n) q \|_{L^1} & = \int_{\mathbb{R}^d} (1+x^n) \text{sgn}(q) \partial_t q dx \\
	& = \underbrace{\int_{\mathbb{R}^d} (1+x^n) \text{sgn}(q) \Delta q dx}_{a} + \underbrace{\int_{\mathbb{R}^d} (1+x^n) \text{sgn}(q) \nabla (\nabla U q) dx}_{b}.
	\end{split}
	\end{equation*}
	Integration by parts applied to the term $(b)$ yields
	\begin{subequations}
		\begin{align*}
		\int_{\mathbb{R}^d} (1+x^n) \text{sgn}(q) \nabla (\nabla U q) dx & = \nabla U q (1+x^n) \text{sgn}(q) \bigg |_{-\infty}^{+\infty} - \int_{\mathbb{R}^d} \nabla U q\  \text{sgn}(q) n x^{n-1} dx \\
		& \leq - \gamma \|(1+x^{3+(n-1)}) q \|_{L^1} + C(\beta) \|q\|_{L^1}, 
		\end{align*}
	\end{subequations}
	for some constants $\gamma > 0$ and $C(\beta)>0$ depending on the modulus of $\beta$, where we used the dissipation condition for the shifted potential~\eqref{eq:dissipation_assumption}. Similarly, integration by parts applied to the term $(a)$ yields
	\begin{subequations}
		\begin{align*}
		\int_{\mathbb{R}^d} (1+x^n) \text{sgn}(q) \Delta q dx & = - \int_{\mathbb{R}^d} (\nabla q) \text{sgn}(q) n x^{n-1} dx \\
		& \leq - n \int_{\mathbb{R}^d} \nabla \|q\| x^{n-1} dx \\
		& = n(n-1) \int_{\mathbb{R}^d} \|q\| x^{n-2} dx = n(n-1) \|x^{n-2} q \|_{L^1}.
		\end{align*}
	\end{subequations}
	Hence,
	\begin{equation*}
	\frac{d}{dt} \|(1+x^n) q \|_{L^1} + \gamma \|(1+x^{3+(n-1)}) q \|_{L^1} \leq C(\beta) \|q\|_{L^1}
	\end{equation*}
	Setting $\gamma = 0$, integrating with respect to time, using the $L^1$-estimate \eqref{L1_estimate} and noticing that
	\begin{equation*}
	\|(1+x^n) q(\cdot,0;\beta) \|_{L^1} \geq \|q(\cdot,0;\beta)\|_{L^1},
	\end{equation*}
	we finally obtain inequality \eqref{weighted_L1_estimate}.
\end{proof}
\end{lemma}

Next, given a weighted weak solution of our non-autonomous Fokker-Planck equation, we establish a $L^1$-localization estimate.

%--------------------------------------- Lemma: L^1 localization estimate --------------------------------------%
\begin{lemma}[$L^1$ localization estimate] \label{lem:loc_estimate}
The unique weak solution of the non-autonomous Fokker Planck equation~\eqref{eq:non autonomous FP gradient} satisfies for any given $\beta \in C^{1/2}(\mathbb{R},\mathbb{R}^d)$ the $L^1$ localization estimate
\begin{equation}
\|(1+x^n) q(\cdot,t;\beta) \|_{L^1} \leq C(\beta) \frac{1+t^N}{t^N}
\|q(\cdot,0;\beta)\|_{L^1},
\label{L1_localization}
\end{equation}
for any $n\ \in \mathbb{N}, t>0$, some $N \in \mathbb{N}$ and some constant $C=C(\beta)<\infty$. 
\begin{proof}
Let us consider again the inequality
\begin{equation*}
\frac{d}{dt} \|(1+x^n) q\|_{L^1} + \gamma \|(1+x^{3+(n-1)}) q \|_{L^1} \leq C(\beta) \|q\|_{L^1}
\end{equation*}
for any $n \in \mathbb{N}$ and some constant $\gamma>0$, as derived in the proof of Lemma~\ref{lem:weighted_L1_estimate}. We multiply both sides by $t^N$ for some $N \in \mathbb{N}$:
\begin{equation*}
t^N \frac{d}{dt} \|(1+x^n) q \|_{L^1} + t^N \gamma \|(1+x^{3+(n-1)}) q \|_{L^1} \leq t^N C(\beta) \|q\|_{L^1}.
\end{equation*}
Using the H{\"o}lder inequality
\begin{equation}
t^{N-1} \|(1+x^n) q \|_{L^1} \leq \|q\|^{\frac{1}{N}}_{L^1} (t^{N} \|(1+x^{n \frac{N}{N-1}}) q \|_{L^1})^{\frac{N-1}{N}},
\label{eq:interpol_inequality_1}
\end{equation}
we obtain
\begin{equation*}
\frac{d}{dt} (t^N \|(1+x^n) q\|_{L^1}) + \gamma t^N \|(1+x^{3+(n-1)}) q \|_{L^1} \leq C(\beta) (1+t^N) \|q\|_{L^1},
\end{equation*}
which implies
\begin{equation*}
\frac{d}{dt} (t^N \|(1+x^n)q \|_{L^1}) \leq C(\beta) (1+t^N) \|q\|_{L^1} \leq C (\beta)(1+t^N) \|q_0\|_{L^1},
\end{equation*}
thanks to the $L^1$-estimate~\eqref{L1_estimate}. Integrating on both sides, we conclude.
\end{proof}
\end{lemma}

Next, we prove two smoothing estimates which ensure the unique weak solution of the initial value problem~\eqref{eq:non autonomous FP gradient} belongs to $L^2$ and, in fact, to the Sobolev space $H^1:=W^{1,2}$, for all $t>0$, that is, its spatial derivative belongs to $L^2$ as well.

%-------------------------------------- Lemma: First smoothing estimate ----------------------------------------%
\begin{lemma}[First smoothing estimate] \label{lem:first_smoothing_estimate}
	The unique weak solution of the non-autonomous Fokker Planck equation~\eqref{eq:non autonomous FP gradient} satisfies for any given $\beta \in C^{1/2}(\mathbb{R},\mathbb{R}^d)$
	\begin{equation}
	\|q(\cdot,t;\beta)\|^2_{L^2} + \int_{t}^{t+1} \|\nabla_s q(s,t;\beta) \|^2_{L^2} ds \leq C(\beta) \frac{t^N+1}{t^N} \|q(\cdot,0;\beta) \|^2_{L^1},
	\label{first_smoothing_estimate}
	\end{equation}
	for some $N\ \in \mathbb{N}$ and some constant $C=C(\beta)<\infty$. 
	\begin{proof}
		Let us multiply the non-autonomous Fokker Planck equation \eqref{eq:non autonomous FP gradient} by $q$ and integrate over $\mathbb{R}^d$:
		\begin{equation*}
		\int_{\mathbb{R}^d} q \partial_t q dx = \int_{\mathbb{R}^d} q \Delta q dx + \int_{\mathbb{R}^d} q \nabla (\nabla U q) dx.
		\end{equation*}
		Using integration by parts we obtain
		\begin{subequations}
			\begin{align*}
			\int_{\mathbb{R}^d} q \Delta q dx & = \nabla q \nabla q \bigg |_{-\infty}^{+\infty} - \int_{\mathbb{R}^d} (\nabla q)^2 dx = - \|\nabla q\|^2_{L^2} \\
			\int_{\mathbb{R}^d} q \nabla (\nabla U q) dx & =  \frac{1}{2} \int_{\mathbb{R}^d} \|q\|^2 \Delta U dx.
			\end{align*}
		\end{subequations}
		Putting everything together,
		\begin{equation*}
		\frac{1}{2} \frac{d}{dt} \|q\|^2_{L^2} + \|\nabla q\|^2_{L^2} = \frac{1}{2} \langle \Delta U, \|q\|^2 \rangle.
		\end{equation*}
		Thanks to the dissipation condition~\eqref{eq:dissipation_assumption} on the potential $U$, we obtain
		\begin{equation*}
		\frac{1}{2} \langle \Delta U, q^2 \rangle \leq C(\beta) \langle 1+\|x\|^2, q^2 \rangle \leq C(\beta) \|(1+\|x\|^2) q \|_{L^1} \|q \|_{L^{\infty}}.
		\end{equation*}
		Next, we use the inequality
		\begin{equation}
		\|q\|^2_{L^\infty} \leq \|q\|_{L^2} \|\nabla q\|_{L^2}
		\label{eq:interpol_inequality_2}
		\end{equation}
		and deduce
		\begin{equation*}
		\begin{split}
		\frac{1}{2} \frac{d}{dt} \|q\|^2_{L^2} + \|\nabla q \|^2_{L^2} & \leq C(\beta) \|(1+\|x\|^2) q \|_{L^1} \|q\|_{L^\infty} \\
		& \leq C(\beta) \|(1+\|x\|^2) q \|_{L^1} \|q\|^{1/2}_{L^2} \| \nabla q\|^{1/2}_{L^2}.
		\end{split}
		\end{equation*}
		Using also $a b \leq \frac{1}{2} (\epsilon^2 a^2 + \epsilon^{-2} b^2)$ for $\epsilon$ small enough, we obtain 
		\begin{equation*}
		\frac{d}{dt} \|q\|^2_{L^2} + \|\nabla q\|^2_{L^2} \leq C(\beta) \|q\|^2_{L^2} + C(\beta) \|(1+\|x\|^2) q \|^2_{L^1}.
		\end{equation*}
		Using inequality \eqref{eq:interpol_inequality_2}, we also have
		\begin{equation*}
		\begin{split}
		\|q\|^2_{L^2} & \leq \|q\|_{L^1} \|q\|_{L^\infty} \\
		&  \leq \|q\|_{L^1} \|q\|^{1/2}_{L^2} \|\nabla q\|^{1/2}_{L^2},
		\end{split}
		\end{equation*}
		from which we deduce
		\begin{equation}\label{eq:interpol_inequality_3}
		\|q\|_{L^2} \leq \|q\|^{2/3}_{L^1} \|\nabla q\|^{1/3}_{L^1}.
		\end{equation}
		Putting everything together, we obtain the relationship
		\begin{equation*}
		\frac{d}{dt} \|q \|^2_{L^2} + \|\nabla q \|^2_{L^2} \leq C(\beta) \|(1+x^2) q \|_{L^1}.
		\end{equation*}
		Multiplying by $t^N$ for some $N \in \mathbb{N}$, using again inequality~\eqref{eq:interpol_inequality_3} and integrating with respect to time from $t$ to $t+1$, we conclude.
	\end{proof}
\end{lemma}

%----------------------------------- Lemma: Second smoothing estimate ------------------------------------%	
\begin{lemma}[Second smoothing estimate] \label{lem:second_smoothing_estimates}
	The unique weak solution of the non-autonomous Fokker Planck equation~\eqref{eq:non autonomous FP gradient} satisfies for any given $\beta \in C^{1/2}(\mathbb{R},\mathbb{R}^d)$
	\begin{equation}
	\|q(\cdot,t;\beta) \|_{H^1} \leq C(\beta) \frac{1+t^N}{t^N} \|q(\cdot,0;\beta) \|_{L^1}
	\label{second_smoothing_estimate}
	\end{equation}
	for some $N\ \in \mathbb{N}$ and some constant $C=C(\beta)<\infty$. 
	\begin{proof}
		Let us multiply the non-autonomous Fokker Planck equation~\eqref{eq:non autonomous FP gradient} by $\Delta q$ and integrate over $\mathbb{R}^d$:
		\begin{equation}\label{eq:fun_weak_sol_6}
		\underbrace{\int_{\mathbb{R}^d} \partial_t q \Delta q dx}_{a} = \underbrace{\int_{\mathbb{R}^d} \left(\Delta q \right)^2 dx}_{b} + \underbrace{\int_{\mathbb{R}^d} \Delta q \nabla (\nabla U q) dx}_{c}.
		\end{equation}
		Terms $(a)$ and $(b)$ in~\eqref{eq:fun_weak_sol_6} can be rewritten respectively as
		\begin{subequations}
			\begin{align*}
			\int_{\mathbb{R}^d} \partial_t q \Delta q dx & = - \frac{1}{2} \frac{d}{dt} \|\nabla q \|^2_{L^2} \\
			\int_{\mathbb{R}^d} \left(\Delta q \right)^2 dx & = \| \Delta q \|^2_{L^2}.
			\end{align*}
		\end{subequations}
		Regarding term $(c)$, using again integration by parts and vanishing at the boundary due to the dissipation condition, we deduce
		\begin{subequations}
			\begin{align*}
			\int_{\mathbb{R}^d} \Delta q \nabla (\nabla U q) dx & = - \int_{\mathbb{R}^d} \nabla q \nabla (\nabla (\nabla U q)) dx \\
			& = - \int_{\mathbb{R}^d} \nabla q \left(\nabla^3 U q + 2 \Delta U \nabla q + \nabla U \Delta q \right) dx \\
			& = - \int_{\mathbb{R}^d} \nabla q \nabla^3 U q dx - 2 \int_{\mathbb{R}^d} \nabla q \nabla q \Delta U dx - \int_{\mathbb{R}^d} \nabla q \nabla U \Delta q dx.
			\end{align*}
		\end{subequations}
		Using
		\begin{equation*}
		\int_{\mathbb{R}^d} \Delta U \nabla q \nabla q dx = - \int_{\mathbb{R}^d} q \nabla^3 U \nabla q  dx - \int_{\mathbb{R}^d} \Delta q \Delta U q dx,
		\end{equation*}
		we obtain
		\begin{equation*}
		\int_{\mathbb{R}^d} \Delta q \nabla (\nabla U q) dx = - \int_{\mathbb{R}^d} \Delta U \nabla q \nabla q dx + \int_{\mathbb{R}^d} \Delta q \Delta U q dx - \int_{\mathbb{R}^d} \nabla q \nabla U \Delta q dx.
		\end{equation*}
		We further notice that
		\begin{equation*}
		\int_{\mathbb{R}^d} \nabla q \Delta q \nabla U dx = - \frac{1}{2} \int_{\mathbb{R}^d} \nabla q \nabla q \Delta U dx.
		\end{equation*}
		Putting everything together, we find
		\begin{subequations}
			\begin{align*}
			\int_{\mathbb{R}^d} \nabla (\nabla U q) \Delta q dx & = - \frac{1}{2} \int_{\mathbb{R}^d} \Delta U \nabla q \nabla q dx + \int_{\mathbb{R}^d} \Delta q \Delta U q dx \\
			& = \frac{3}{2} \int_{\mathbb{R}^d} \Delta q \Delta U q dx + \frac{1}{2} \int_{\mathbb{R}^d} \nabla^3 U \nabla q q dx.
			\end{align*}
		\end{subequations}
		Concisely, the equality above can be written as
		\begin{equation*}
		\langle \nabla (\nabla U q), \Delta q \rangle_{L^2} = \frac{3}{2} \langle \Delta q, \Delta U q \rangle_{L^2} + \frac{1}{2} \langle \nabla^3 U \nabla q, q \rangle_{L^2}.
		\end{equation*}
		Therefore, using the inequality
		\begin{equation}\label{eq:interpol_inequality_4}
		\|\nabla q \|^2_{L^2} \leq \|\Delta q \|_{L^2} \|q\|^2_{L^2}
		\end{equation}
		we deduce
		\begin{equation*}
		\begin{split}
		|\langle \nabla (\nabla U q), \Delta q \rangle_{L^2}| & \leq \epsilon \| \Delta q \|^2_{L^2} + C(\beta) \| \nabla q \|^2_{L^2} + C(\beta) \| (1+x^2) q \|^2_{L^2} \\
		& \leq \epsilon \| \Delta q \|^2_{L^2} + C(\beta) \|\nabla q\|^2_{L^2} + C(\beta) \|(1+x^4) q \|_{L^1} \|q\|_{L^{\infty}} \\
		& \leq \epsilon \|\Delta q \|^2_{L^2} + C(\beta) \left( \|\nabla q \|^2_{L^2} + \|q\|^2_{L^2} + \|(1+x^4) q \|_{L^1} \|q\|^2_{L^\infty} \right) \\
		& \leq 2 \epsilon \|\Delta q \|^2_{L^2} + C(\beta) \left( \|\nabla q\|^2_{L^2} + \|q\|^2_{L^2} + \|(1+x^4) q \|_{L^1} \|q\|^2_{L^\infty} \right)
		\end{split}
		\end{equation*}
		for $\epsilon>0$ small enough and some positive constant $C=C(\beta)$. This implies
		\begin{equation*}
		\frac{d}{dt} \|\nabla q \|^2_{L^2} + \gamma \|\Delta q \|^2_{L^2} + \|\nabla q \|^2_{L^2} \leq C(\beta) \left( \|\nabla q\|^2_{L^2} + \|q\|^2_{L^2} + \|(1+x^4) q \|_{L^1} \|q\|_{L^\infty} \right),
		\end{equation*}
		for some constant $\gamma>0$. Multiplying by $t^N$ for some $N \in \mathbb{N}$, using again inequality~\eqref{eq:interpol_inequality_4}, integrating with respect to time and, finally, employing the first smoothing estimate~\eqref{first_smoothing_estimate}, we conclude.
	\end{proof}
\end{lemma}

Finally, we are ready to prove the main theorem of this section. Having established existence, uniqueness and all estimates above for signed measures, we now restrict ourselves to probability measures. We denote by $H^k:=W^{k,2}$ the Hilbert space of all functions $f \in L^2$ such that their weak derivatives up to order $k$ have finite $L^2$ norm. We further denote by $L^1_{1+x^n}$, for any $n \in \mathbb{N}$, the weighted space of measurable functions $f$ such that
\begin{equation*}
\|f\|_{L^1_{1+x^n}} := \int_{\mathbb{R}^d} (1+x^n) \|f(x)\| dx < \infty.
\end{equation*}

%-------------------------------- Thm: Existence uniqueness regularity ---------------------------------------%	
\begin{proof}[Proof of Theorem~\ref{thm:existence_uniqueness_regularity}]
Combining the $L^1$ estimates \eqref{L1_estimate}, \eqref{weighted_L1_estimate}, \eqref{L1_localization} and the smoothing estimates \eqref{first_smoothing_estimate}, \eqref{second_smoothing_estimate}, we conclude that for any given $\beta \in C^{1/2}(\mathbb{R},\mathbb{R}^d)$, $q(t) \in L^1_{(1+x^n)} \cap H^1$ for any $n \in \mathbb{N}$ and $t>0$. In order to gain more regularity, we differentiate the Fokker-Planck equation~\eqref{eq:non autonomous FP gradient} with respect to the space variable $x$ iteratively, via a standard bootstrapping procedure. After one step, we obtain $\partial_x q(t) \in L^1_{(1+x^n)} \cap H^1$, which implies in particular $q(t) \in H^2$ for $t>0$. After two steps, we obtain $\partial^2_x q(t) \in L^1_{(1+x^n)} \cap H^1$, which implies $q(t) \in H^3$ for $t>0$ and so on. This shows that, at any time $t>0$, $q(t) \in H^{\infty}$ and any spatial derivative $\partial_x^m q(t)$ decays faster than any polynomial as $|x| \to \infty$, at any time $t>0$. Finally, we recall that the Sobolev space $H^s(\mathbb{R}^d)$ can be continuously embedded into $C^k(\mathbb{R}^d)$ for any $k \in \mathbb{N}$ and $s>k+\frac{n}{2}$ \cite{cerda2010linear}. This readily implies $H^\infty$ can be continuously embedded in $C^{\infty}$. Putting everything together, we conclude that $q(t)$ belongs to the Schwarz space $\mathcal{S}$ for any time $t>0$.
\end{proof}

%---------------------------------------------------------------------------------------------------------------
%								Proofs of the results from Section 3
%---------------------------------------------------------------------------------------------------------------
\section{Proofs of Propositions~\ref{prop: exp_contraction_kantorovich} and ~\ref{prop:convergence_Wf_L1}} \label{app:contraction_non_auto_FP_proof}

%------------------------------ Lemma dissipation condition and strict convexity ------------------------------%
\begin{lemma}[Dissipation condition and strict convexity]\label{lem:dissipation_condition_strict_convexity}
Assume the potential $V$ satisfies the dissipation condition~\eqref{eq:dissipation_assumption}, that is,
\begin{equation*}
\nabla V(x) \cdot x \|x\|^2 \geq \frac{1}{2} \|x\|^6 - C
\end{equation*}
for some $C>0$. Then $V$ is strictly convex outside a given ball in $\mathbb{R}^d$.
\begin{proof}
We rewrite the dissipation condition as
\begin{equation}
\label{eq:fun_grad_potential}
\nabla V(x) = M x \|x\|^2 + h(x),
\end{equation}
with $M \geq \frac{1}{2}, h(x) = O(\|x\|^2)$ and such that
\begin{equation*}
f(x) := \left(M-\frac{1}{2}\right) \|x\|^6 + h(x) \cdot x \|x\|^2
\end{equation*}
is lower bounded. Differentiating~\eqref{eq:fun_grad_potential} yields
\begin{equation*}
\Delta V(x) = 3 M \|x\|^2 + g(x),
\end{equation*}
with $g(x) = \nabla \cdot h(x) = O(\|x\|)$, meaning that there exists $L>0$ and $x_0 \in \mathbb{R}$ such that for any $\|x\| \geq x_0$
\begin{equation*}
|g(x)| \leq L \|x\|.
\end{equation*}
Then, letting $N:=3 M$, we have
\begin{equation*}
\Delta V (x) = N \|x\|^2 + g(x) \geq N \|x\|^2 - L \|x\|
\end{equation*}
and so the potential $V$ is strictly convex outside the ball centered in $\frac{LN}{2}$ with radius $\frac{LN}{2}$.
\end{proof}
\end{lemma}
We illustrate this result in the context of the example of Section~\ref{subsubsec:double_well}. % satisfies the dissipation condition and is strictly convex outside a given ball.
%--------------------------------------- Example (double well potential) ---------------------------------------%
\begin{example}[Double well potential]
We consider the one-dimensional double-well potential $V'(x) = x(x^2-a)$, with $a>0$. Then, the dissipation condition~\eqref{eq:dissipation_assumption} is fulfilled if and only if
\begin{equation*}
x^4 \left(\frac{1}{2} x^2 - a \right) \geq - C
\end{equation*}
for some constant $C>0$. Let $f(x) := \frac{1}{2} x^6 - x^4 a$. Then, we set
\begin{equation*}
f'(x) = 3 x^5 - 4 a x^3 = x^3 (3 x^2 - 4 a) = 0
\end{equation*}
and we find the local extrema $x=0$ and $x \pm \sqrt{\frac{4 a}{3}}$. Hence, $f(0) = 0$ and
\begin{equation*}
f \left(\pm \sqrt{\frac{4 a}{3}} \right) = a^3 \left(-\frac{16}{27} \right) \geq - C
\end{equation*}
for any constant $C \geq \frac{16}{27} a^3$. The dissipation condition is therefore satisfied. Moreover, 
\begin{equation*}
V''(x) = 3 x^2 -a > 0 \iff x^2 > \frac{a}{3}
\end{equation*}
and so we immediately deduce the potential $V$ is strictly convex outside the ball centred at $0$ with radius $\sqrt{\frac{a}{3}}$. 
\end{example}
The converse implication of Lemma~\eqref{lem:dissipation_condition_strict_convexity} does not hold, as the following example shows.
%---------------- Example (strict convexity outside a ball does not imply the dissipation condition) ----------%
\begin{example}[Strict convexity outside a ball does not imply the dissipation condition]
We consider the one-dimensional potential
\begin{equation*}
V(x) = \frac{x^4}{8} - a \frac{x^2}{2}, \ \ a>0.
\end{equation*}
Then,
\begin{equation*}
\begin{split}
V'(x) & = \frac{x^3}{2} - a x = x \left(\frac{x^2}{2} - a \right) \\
V''(x) & = \frac{3}{2} x^2 -a > 0 \iff x^2 > \frac{2 a}{3}.
\end{split}
\end{equation*}
We immediately see that $V$ is strictly convex outside the ball centred at $0$ with radius $\sqrt{\frac{2 a}{3}}$. However, the dissipation condition reads as
\begin{equation*}
V'(x) x^3 = \frac{1}{2} x^6 - a x^4 \geq \frac{1}{2} x^6 - C
\end{equation*}
which is equivalent to $C \geq a x^4$ for some positive constant $C$ and this is clearly not possible.
\end{example}

%-------------------------------Proof: Prop exponential contraction Kantorovich metric --------------------------%
\begin{proof}[Proof of Proposition~\ref{prop: exp_contraction_kantorovich}]
We extend results by Eberle \cite{eberle2016reflection} to a non-autonomous setting. Consider the difference process $z(t):= x(t) - y(t)$. Then, 
\begin{equation*}
d z(t) = (\nabla V(y) - \nabla V(x)) dt + 2 |\sigma^{-1} z(t)|^{-1} z(t) d \tilde{W}(t), \ \ t<T
\end{equation*}
and $z(t) =0$ for $t \geq T$, where
\begin{equation*}
\tilde{W}(t):= \int_0^t e^\top(s) d W(s)
\end{equation*}
is a new Brownian motion by Levy's characterization, $e(t)$ is the unit vector defined by
\begin{equation*}
e(t):=\frac{\sigma^{-1}(x(t)-y(t))}{|\sigma^{-1}(x(t)-y(t))|}
\end{equation*}
and $T$ is the coupling time. Next, we define $r(t):= \|z(t)\|:=|\sigma^{-1} z(t)|$. By application of Ito's formula we find
\begin{equation*}
d r(t) = 2 |\sigma^{-1} z(t) |^{-1} r(t) d \tilde{W}(t) + r^{-1}(t) z(t) (\sigma \sigma^T)^{-1} (-\nabla V(x(t)) + \nabla V(y(t))) dt.
\end{equation*}
Given a smooth function $f \in C^1(\mathbb{R}^d)$, this implies
\begin{equation}\label{eq:dynamics_f_r}
\begin{split}
d f(r(t)) & = 2 |\sigma^{-1} z(t)|^{-1} r(t) f'(r(t)) d \tilde{W}(t) \\
& \quad + r^{-1}(t) z(t) (\sigma \sigma^T)^{-1}(-\nabla V(x(t)) + \nabla V(y(t))) f'(r(t)) dt \\
& \quad \quad + 2 |\sigma^{-1} z(t)|^{-2} r^2(t) f''(r(t)) dt.
\end{split}
\end{equation}
We also define for any $r>0$ the function
\begin{equation}\label{eq:function_k}
\begin{split}
k(r) & := \inf_{x,y\in\mathbb{R}^d,~\|x-y\|=r} \left\{ -2 \frac{|\sigma^{-1} (x-y)|^2}{\|x-y\|^2} \frac{(x-y) \cdot (\sigma \sigma^T)^{-1} (-\nabla V(x) + \nabla V(y))}{\|x-y\|^2} \right\} \\
& = \inf_{x,y\in\mathbb{R}^d,~\|x-y\|=r} \left\{ -2 \frac{(x-y) \cdot (\sigma \sigma^T)^{-1} (-\nabla V(x) + \nabla V(y))}{\|x-y\|^2} \right\}.
\end{split}
\end{equation}
Indeed, $k(r)$ is the largest positive real number such that
\begin{equation*}
(x-y) \cdot (\sigma \sigma^T)^{-1} (-\nabla V(x) + \nabla V(y)) \leq -\frac{1}{2} k(r) \|x-y\|^2
\end{equation*}
for any $x,y \in \mathbb{R}^d$ with $\|x-y\|=r$. Let us denote by $m(t)$ the drift on the right hand side of~\eqref{eq:dynamics_f_r}. By definition of $k$, 
\begin{equation*}
m(t) \leq \Gamma(t) := 2 |\sigma^{-1} z(t)|^{-2} r^2(t) \cdot \left(f''(r(t)) - \frac{1}{4} k(r(t)) f'(r(t))\right).
\end{equation*}
Hence, the process $e^{ct} f(r(t))$ is a supermartingale for $t<T$ if $\Gamma(t) \leq - c f(r(t))$. We aim to find a constant $c$ and function $f$ such that this inequality holds. Let
\begin{equation*}
\alpha:= \sup \left\{|\sigma^{-1}z|^2 : z \in \mathbb{R}^d \ \ \text{with}\ \ \|z\|=1 \right\}.
\end{equation*}
Since for any $z \in \mathbb{R}^d$
\begin{equation*}
|\sigma^{-1} z |^2 \leq \alpha \|z\|^2,
\end{equation*}
it suffices for $f$ to satisfy
\begin{equation}
f''(r) - \frac{1}{4} r k(r) f'(r) \leq - \frac{\alpha c}{2} f(r)
\label{eq:suff_condition_f_second}
\end{equation}
for all $r>0$, cf.~\cite[eq. 63]{eberle2016reflection}. We observe this equation holds with $c=0$ in case 
\begin{equation*}
f'(r) = \varphi(r) := \exp\left(-\frac{1}{4} \int_0^r s k^{-}(s) ds\right),
\end{equation*}
where $k^{-}:=\max \left\{-k,0\right\}$ denotes the negative part of the function $k$. Next, following \cite{eberle2016reflection},  we make the ansatz
\begin{equation}
f'(r) = \varphi(r) g(r),
\label{eq:ansatz}
\end{equation}
where $g \geq \frac{1}{2}$ is a decreasing absolutely continuous function satisfying $g(0)=1$. Notice that the condition $g \geq 0$ is necessary to ensure that $f$ is non-decreasing. The condition
\begin{equation*}
\frac{1}{2} \geq g \geq 1
\end{equation*}
ensures
\begin{equation*}
\frac{\varPhi}{2} \leq f \leq \varPhi, \ \ \varPhi(r):= \int_0^r \varphi(s) ds.
\end{equation*}
The ansatz \eqref{eq:ansatz} yields
\begin{equation*}
f''(r) = -\frac{1}{4} k^{-}(r) f(r) + \varphi(r) g(r) \leq \frac{1}{4} r k(r) f(r) + \varphi(r) g'(r).
\end{equation*}
In turn, condition \eqref{eq:suff_condition_f_second} is satisfied if
\begin{equation}
g'(r) \leq -\frac{\alpha c}{2} \frac{f(r)}{\varphi(r)}.
\label{eq:condition_g_prime}
\end{equation}
Next, we define two constants $R_0,R_1 \geq 0$, with $R_0 \leq R_1$:
\begin{align*}
R_0 & := \inf \left\{R \geq 0: k(r) \geq 0, \forall  r \geq R \right\} \\
R_1 & := \inf \left\{R \geq R_0: k(r) R (R-R_0) \geq 8, \forall r \geq R \right\}.
\end{align*}
As remarked in~\cite{eberle2016reflection}, we can rewrite $k$ as
\begin{equation*}
k(r) = \inf \left\{2 \int_0^1 \partial^2_{(x-y)/|x-y|} (\sigma \sigma^\top)^{-1} V((1-t)x + ty) dt: x,y \in \mathbb{R}^d s.t. |x-y|=r \right\}.
\end{equation*}
Thanks to lemma~\ref{lem:dissipation_condition_strict_convexity}, the potential $V$ is strictly convex outside a given ball in $\mathbb{R}^d$ and this, in turn, ensures $k$ is continuous on $(0,\infty)$ and such that
\begin{equation*}
\lim_{r \to +\infty} \inf k(r) >0, \ \ \ \int_0^1 r k^{-}(r) dr < \infty.
\end{equation*}
Thanks to this result, both constants $R_0,R_1$ are finite. For $r \geq R_1$, condition \eqref{eq:suff_condition_f_second} is satisfied since $k$ is sufficiently positive. It is then enough to assume condition \eqref{eq:condition_g_prime} holds on the open interval $(0,R_1)$. Under this assumption,
\begin{equation}
g(R_1) \leq 1 - \frac{\alpha c}{2} \int_0^{R_1} f(s) \varphi^{-1}(s) ds \leq 1 - \frac{\alpha c}{4} \int_0^{R_1} \varPhi(s) \varphi^{-1}(s) ds.
\label{eq:condition_g_R1}
\end{equation}
Condition \eqref{eq:condition_g_R1}, in turn, is satisfied if
\begin{equation*}
\alpha c \leq \frac{2}{\int_0^{R_1} \varPhi(s) \varphi^{-1}(s) ds}.
\end{equation*}
So, by choosing, for $r<R_1$,
\begin{equation*}
g'(r) = - \frac{\varPhi(r)}{2 \varphi(r)} \bigg/ \int_0^{R_1} \frac{\varPhi(s)}{\varphi(s)} ds,
\end{equation*}	
condition \eqref{eq:condition_g_prime} is fulfilled if we choose the constant as
\begin{equation*}
\alpha c = 1 \big/ \int_0^{R_1} \varPhi(s) \varphi^{-1}(s) ds.
\end{equation*}
At this point, we can show that the quantity $\Gamma$ is smaller than $-c f(r)$, with our choices of $f$ and $c$. Consider the scenario $r<R_1$. Then, we have (see~\cite[eq. (68)]{eberle2016reflection})
\begin{equation}
f''(r) \leq \frac{1}{4} r k(r) f'(r) - \frac{1}{2} f(r) \bigg/ \int_0^{R_1} \varPhi(s) \varphi^{-1}(s) ds.
\label{eq:f_second_a}
\end{equation}
Consider, now, the scenario $r \geq R_0$. Then, 
\begin{equation*}
f'(r) = \frac{\varphi(r)}{2} = \frac{\varphi(R_0)}{2},
\end{equation*}
and $k(r) R_1 (R_1 - R_0) \geq 8$ by construction of $R_1$. Moreover, we know that $r \geq R_0$, the function $\varphi$ is constant and, therefore $\varPhi(r) = \varPhi(R_0) + (r-R_0) \varphi(R_0)$. Also,
\begin{equation*}
\int_{R_0}^{R_1} \varPhi(s) \varphi^{-1}(s) ds \geq (R_1-R_0) \varPhi(R_1) \varphi^{-1}(R_0)/2.
\end{equation*}
This implies (see~\cite[eq. (69)]{eberle2016reflection})
\begin{equation}
f''(r) - \frac{1}{4} r k(r) f'(r) \leq - \frac{1}{2} f(r) \bigg/ \int_0^{R_1} \varPhi(s) \varphi^{-1}(s) ds.
\label{eq:f_second_b}
\end{equation}
Putting together equations \eqref{eq:f_second_a} and \eqref{eq:f_second_b}, we conclude the key relationship
\begin{equation*}
\Gamma(t) \leq - c f(r(t))
\end{equation*}
at all times $t<T$. For any coupling $\gamma_t$ of the process $(x(t),y(t))$, we take the expectation on both sides of~\eqref{eq:dynamics_f_r} and obtain
\begin{equation}\label{eq:ode_expectation}
\mathbb{E}^{\gamma_t}[f(r(t))] = \mathbb{E}^{\gamma_t}[f(r(s))] + \int_s^t \mathbb{E}^{\gamma_t}[m(u)] du
\end{equation}
for any $s \leq t<T$. Let $\Upsilon(t):=\mathbb{E}^{\gamma_t}[f(r(t))]$. Then, differentiating~\eqref{eq:ode_expectation} with respect to time yields
\begin{equation*}
\Upsilon'(t) = \mathbb{E}^{\gamma_t}[m(t)].
\end{equation*}
Since we have proved that $m(t) \leq \Gamma(t) \leq - c f(r(t))$ for $t<T$, we deduce
\begin{equation}
\Upsilon'(t) \leq - c \Upsilon(t).
\end{equation}
Thanks to standard Gronwall's lemma, we deduce
\begin{equation}
\Upsilon(t) \leq \Upsilon(s) e^{-c(t-s)}
\end{equation}
for all $s \leq t \leq T$. Hence, $t \to e^{ct} \mathbb{E}^{\gamma_t}[d_f(x(t),y(t))]$ is a decreasing function of time. This key result implies
\begin{align*}
\mathcal{W}_{f}(\mu_{t,\beta}, \nu_{t,\beta}) & \leq \mathbb{E}^{\gamma_{t,\beta}}[d_f (x(t), y(t))] \leq e^{-ct} \mathbb{E}^{\gamma_{t,\beta}}[d_f (x_0, y_0)],
\end{align*}
where $\mu_{t,\beta}$ and $\nu_{t,\beta}$ denote the time-$t$ evolved probability measures of the process $x(t)$ with respect to the initial distributions $\mu$ and $\nu$ respectively, and $\gamma_{t,\beta}$ denotes their coupling, given a realization of $\beta$. Taking the infimum over all couplings $\gamma_{t,\beta}$, we conclude.
\end{proof}

%-------------------------------Proof: Prop convergence in L^1 from W^f --------------------------------%
\begin{proof}[Proof of Proposition~\ref{prop:convergence_Wf_L1}]
By construction, the function $f$ in Proposition~\ref{prop: exp_contraction_kantorovich} is concave, increasing and satisfies $f(0)=1, f'(0)=1$. This implies that $f'(x) x \leq f(x) \leq x$. Moreover, $\frac{\varphi(R_0)}{2} \leq f' \leq 1$ thanks to the properties of $\varphi$ and $g$. Hence,
\[
\frac{\varphi(R_0)}{2} \|x-y\| \leq d_f(x,y) \leq \|x-y\|
\]
for any $x,y \in \mathbb{R}^d$. For any coupling $\gamma_{t,\beta}$ of $\mu_{t,\beta}$ and $\nu_{t,\beta}$, 
\begin{equation*}
\frac{\varphi(R_0)}{2} \mathbb{E}^{\gamma_{t,\beta}}[\|x(t) - y(t)\|]  \leq \mathbb{E}^{\gamma_{t,\beta}}[d_f(x(t),y(t))] \leq e^{-ct} \mathbb{E}^{\gamma_{t,\beta}}[d_f(x_0,y_0)] \leq e^{-ct} \mathbb{E}^{\gamma_{t,\beta}}[\|x_0-y_0\|].
\end{equation*}
Let $K:=2 \varphi(R_0)^{-1}$. Taking the infimum over all couplings $\gamma_{t,\beta}$ yields
\begin{equation*}
K \mathcal{W}^1(\mu_{t,\beta}, \mu_{t+\tau,\beta}) \leq \mathcal{W}_f(\mu_{t,\beta}, \mu_{t+\tau,\beta})
\end{equation*}
for all $t>0$. Hence, if $(\mu_{t,\beta})_{t>0}$ is a Cauchy sequence with respect to $\mathcal{W}_f$, it will be a Cauchy sequence with respect to $\mathcal{W}^1$ as well. Moreover, with $p_{t,\beta}$ denoting the Lebesgue density of the measure $\mu_{t,\beta}$, the Hardy-Landau-Littlewood inequality  \cite{bogachev2015lower,bogachev2015estimates,bogachev2021sobolev} entails
\begin{equation*}
\|p_{t+\tau,\beta} - p_{t,\beta}\|^2_{L^1} \leq C \|\nabla (p_{t+\tau,\beta} - p_{t,\beta})\|_{L^1} \mathcal{W}^1(\mu_{t,\beta},\mu_{t+\tau,\beta})
\end{equation*}
for some constant $C>0$. Since the $L^1$ norm of the gradient is bounded (see Appendix~\ref{app:non_auto_FP_IVP_proof}, Lemma~\ref{lem:second_smoothing_estimates}), we have
\begin{equation*}
\|p_{t+\tau,\beta} - p_{t,\beta}\|^2_{L^1} \leq \bar{C} \mathcal{W}^1(\mu_{t,\beta},\mu_{t+\tau,\beta})
\end{equation*}
for some constant $\bar{C}>0$. Hence, $(p_{t,\beta})_{t>0}$ is a Cauchy sequence in $L^1$. Since $L^1$ is complete, the sequence converges in $L^1$, that is, for any initial condition $\mu \in \mathcal{P}(\mathbb{R}^d)$,
\begin{equation*}
p_\beta = \lim_{t \to \infty} \Phi(t,\beta)p_\mu \in L^1.
\end{equation*}	
\end{proof}

%---------------------------------------------------------------------------------------------------------------
%								Proofs of the results from Section 4
%---------------------------------------------------------------------------------------------------------------
\section{Proofs of results in Section~\ref{sec:from_nonauto_to_stoch}} \label{app:from_nonauto_to_stoch_proof}

%------------------------------------- Proof: Prop: SFPE is a RDS ----------------------------------%
\begin{proof}[Proof of Proposition~\ref{prop:SFPE_RDS}]
	In the stochastic setting, the non-autonomous Fokker-Planck equation~\eqref{eq:non autonomous FP gradient} naturally extends to a \emph{random} Fokker Planck equation in terms of common noise sample paths $\beta$, which we here write in compact form as
	\begin{equation}\label{eq:random_FP}
	\partial_t q = F(\theta_t \beta,q)
	\end{equation}
	for some appropriate functional $F$. Subsequently, in analogy to the discussion in Section~\ref{subsec:derivation_non_auto_FP}, the stochastic Fokker Planck equation~\eqref{stochasticFP} is obtained via the transformation $y=x-\eta \beta$, yielding the analogous form when choosing the stochastic integral to be of Stratonovich type. The cocycle property is obtained from the existence and uniqueness of solutions of \eqref{eq:non autonomous FP gradient} for almost all sample paths, as established in Section~\ref{sec:non_auto_FP_IVP}. The evolution operator $\Phi$ of the random Fokker Planck equation~\eqref{eq:random_FP} is given by  
	\begin{equation*}
	\Phi(t,\beta,q) = q + \int_0^t F(\theta_s \beta, \Phi(s,\beta,q)) ds.
	\end{equation*}
	Following closely the argument in Arnold \cite[Proof of Theorem 2.2.1]{arnold1995random}, we prove the cocycle property (for almost all $\beta \in \Omega_B$). Let $s,t \in \mathbb{R}$ and assume $s>0,t>0$ (the remaining cases are analogous). Then, 
	\begin{equation*}
	\begin{split}
	\Phi(t,\theta_s \beta, \Phi(s,\beta,q)) & = \Phi(s,\beta,q) + \int_0^t F(\theta_{u+s} \beta, \Phi(u,\theta_s \beta, \Phi(s,\beta,q))) du \\
	& = q + \int_0^t F(\theta_s \beta, \Phi(s,\beta,q)) ds \\
	& \quad + \int_{s}^{t+s} F(\theta_z \beta, \Phi(z-s, \theta_s \beta, \Phi(s,\beta,q))) dz,
	\end{split}
	\end{equation*}
	where
	$z=u+s$. Therefore, the function
	\begin{equation*}
	\tilde{\Phi}(u,\beta,q) := \begin{cases}
	\Phi(u,\beta,q) &\text{if $0 \leq u \leq s$}\\
	\Phi(u-s,\theta_s \beta, \Phi(s,\beta,q)) &\text{if $s \leq u \leq s+t$}
	\end{cases}
	\end{equation*}
	satisfies
	\begin{equation*}
	\tilde{\Phi}(t+s,\beta,q) = q + \int_0^{t+s} F(\theta_s \beta, \tilde{\Phi}(u,\beta,q)) du.
	\end{equation*}
	By uniqueness, for $\mathbb{P}$-a.e. $\beta \in \Omega_B$,
	\begin{equation*}
	\Phi(t+s,\beta,q)  = \tilde{\Phi}(t+s,\beta,q) = \Phi(t,\theta_s \beta, \tilde{\Phi}(s,\beta,q)).
	\end{equation*}
\end{proof}

%------------------------------------- Proof: Thm pullback attractor -------------------------------------%
\begin{proof}[Proof of Theorem~\ref{thm:pullback_attractor}]
	From proposition~\ref{prop: exp_contraction_kantorovich} we deduce that there exist a constant $c>0$ and an increasing and convex function $f$ such that for any $t>0, \beta \in C^{1/2}(\mathbb{R},\mathbb{R}^d)$ and initial probability measures $\mu,\nu \in \mathcal{M}(\mathbb{R}^d)$,
	\begin{equation*}
	\mathcal{W}_{f}(\mu_{t,\beta} \nu_{t,\beta}) \leq e^{-c t} \mathcal{W}_{f}(\mu,\nu), 
	\end{equation*}
	with $\mu_{t,\beta}:=\Psi(t,\beta)\mu$ and similarly for $\nu_{t,\beta}$, where $\Psi$ denotes the time-$t$ evolution operator for the measure $\mu$, associated to the time-$t$ evolution operator $\Phi$ of the random Fokker Planck equation~\ref{eq:random_FP}. We show that $(\mu_{t,\beta})_{t>0}$ is a Cauchy sequence in a pullback sense with respect to the $\mathcal{W}_{f}$ metric, that is, $\forall\ \epsilon>0\  \exists\ t>0: \forall\ \tau>0$
	\begin{equation*}
	\mathcal{W}_{f}\left(\Psi(t, \theta_{-t} \beta) \mu,\Psi(t+\tau, \theta_{-(t+\tau)} \beta) \mu \right) < \epsilon.
	\end{equation*}
	Exploiting the pullback operator and the fact that we have a contraction, we deduce
	\begin{align*}
	\mathcal{W}_{f}\left(\Psi(t, \theta_{-t} \beta) \mu,\Psi(t+\tau, \theta_{-(t+\tau)} \beta) \mu \right) & = \mathcal{W}_{f} \left(\Psi(t,\theta_{-t} \beta)\mu, \Psi(t,\theta_{-t}\beta) \circ \Psi(\tau, \theta_{-(t+\tau)}\beta) \mu \right) \\
	& \leq e^{-ct} \mathcal{W}_f \left(\mu, \Psi(\tau,\theta_{-(t+\tau)}\beta) \mu\right).
	\end{align*}
	Then, since $f$ is concave and increasing by construction, 
	\begin{equation*}
	\mathcal{W}_f(\mu,\Psi(\tau,\theta_{-(t+\tau)}\beta) \mu) \leq f(W^1(\mu,\Psi(\tau,\theta_{-(t+\tau)}\beta) \mu)) \leq W^1(\mu,\Psi(\tau,\theta_{-(t+\tau)}\beta) \mu).
	\end{equation*}
	We observe that $C^{1/2}(\mathbb{R},\mathbb{R}^d)$ is a subset of full Wiener measure $\mathbb{P}_\beta$ of the sample path space  $\Omega_B$. Taking the expectation $\mathbb{E}^{\mathbb{P}_\beta}$ with respect to $\mathbb{P}_\beta$ implies
	\begin{equation*}
	\mathbb{E}^{\mathbb{P}_\beta}[\mathcal{W}_f(\mu,\Psi(\tau,\theta_{-(t+\tau)}\beta) \mu)] \leq \mathbb{E}^{\mathbb{P}_\beta}[W^1(\mu,\Psi(\tau,\theta_{-(t+\tau)}\beta) \mu)]
	\end{equation*}
	Let $p_\mu$ and $p_{\tau,t+\tau;\beta}$ denote the Lebesgue densities of $\mu$ and $\Psi(\tau,\theta_{-(t+\tau)}\beta) \mu$ respectively. Their product will be the density of the product measure, which is a simple example of a coupling measure. Therefore,
	\begin{align*}
	\mathbb{E}^{\mathbb{P}_\beta}[W^1(\mu,\Psi(\tau,\theta_{-(t+\tau)}\beta) \mu)] & \leq \mathbb{E}^{\mathbb{P}_\beta} \left[\iint_{\mathbb{R}^{2d}} \|x-y\| p_{\tau,t+\tau;\beta}(x) p_\mu(y) dx dy \right] \\
	& = \iint_{\mathbb{R}^{2d}} \|x-y\| \mathbb{E}^{\mathbb{P}_\beta}[p_{\tau,t+\tau;\beta}(x)] p_\mu(y) dx dy.
	\end{align*}
	Let us define $p_\tau:=\mathbb{E}^{\mathbb{P}_\beta}[p_{\tau,t+\tau;\beta}]$. We notice that this expectation does not depend on $t$ since we are integrating over all $\Omega_B$ and $\theta_{-(t+\tau)} \beta = \theta_{-\tau} \tilde{\beta}$ for some $\tilde{\beta} \in \Omega_B$. Then
	\begin{align*}
	\iint_{\mathbb{R}^{2d}} \|x-y\| p_\tau(x) p_\mu(y) dx dy & = \int_{\mathbb{R}^d} p_\mu(y) \int_{\mathbb{R}^d} \|x-y\| p_\tau(x) dx dy \\
	& \leq \int_{\mathbb{R}^d} p_\mu(y) \left[\int_{\mathbb{R}^d} \|x\| p_\tau(x) + \|y\| p_\tau(x) dx \right] dy \\
	& = \int_{\mathbb{R}^d} p_\mu(y) \left(\int_{\mathbb{R}^d} \|x\| p_\tau(x) dx + \|y\| \right) dy \\
	& = \int_{\mathbb{R}^d} p_\tau(x) \|x\| dx + \int_{\mathbb{R}^d} p_\mu(y) \|y\| dy.
	\end{align*}
	The second term on the RHS of the equation above is bounded. For what concerns the first term, we further notice that $p_\tau$ is the forward solution at time $t=\tau$, with initial condition $p_\mu$ at time $t=0$, of the autonomous Fokker Planck equation for the SDE~\eqref{SDE}
	\begin{equation*}
	\frac{\partial}{\partial t} p = \Delta V(x) p + \nabla V(x) \frac{\partial p}{\partial x} + \frac{1}{2}(\sigma^2 + \eta^2) \frac{\partial^2 p}{\partial x^2}.
	\end{equation*}
	Applying the results from Sections~\ref{sec:non_auto_FP_IVP} and \ref{sec:contraction_non_auto_FP} to the autonomous setting, we deduce this equation admits a unique attractor and, in particular, 
	\begin{equation*}
	\lim_{t \to \infty} p_t = p_\rho\ \ \text{in}\ \ L^1,
	\end{equation*}
	where $p_\rho$ denotes the density of the stationary measure. In fact, we observe that the fixed point $p_\rho$ is invariant under the autonomous evolution operator $\tilde{\Phi} = \Phi(\cdot,0)$, i.e.
	\begin{equation*}
	\tilde{\Phi} p_\rho = p_\rho.
	\end{equation*}
	Thanks to the discussion in Appendix~\ref{app:non_auto_FP_IVP_proof}, we deduce that $\tilde{\Phi}$ maps $L^1$ functions into the Schwartz space $\mathcal{S}$ of rapidly decreasing functions. Therefore, $p_t$ converges exponentially fast to $p_\rho$ as $t \to \infty$ in $\mathcal{S}$. Next, we consider
	\begin{align*}
	\int_{\mathbb{R}^d} p_\tau(x) \|x\| dx & = \int_{\mathbb{R}^d} (p_\tau(x) - p_\rho(x)) \|x\| dx + \int_{\mathbb{R}^d} p_\rho(x) \|x\| dx \\
	& \leq \int_{\mathbb{R}^d} |p_\tau(x) - p_\rho(x)| \|x\|dx + \int_{\mathbb{R}^d} p_\rho(x)\|x\| dx.
	\end{align*}
	For any $\epsilon>0$, there exists $T>0$ such that for any $\tau>T$,
	\begin{equation*}
	\int_{\mathbb{R}^d} |p_\tau(x) - p_\rho(x)| \|x\|dx < \epsilon
	\end{equation*}
	Let us define
	\begin{equation*}
	C := \sup_{\tau \in [0,T]} \int_{\mathbb{R}^d} p_\tau(x) \|x\| dx < \infty,
	\end{equation*}
	which is finite since $\int_{\mathbb{R}^d} p_\tau(x) \|x\| dx$ is finite for any $\tau$ and the supremum is taken over a finite time interval. Then,
	\begin{align*}
	\int_{\mathbb{R}^d} p_\tau(x) \|x\| dx & < \mathbbm{1}_{\tau \leq T} C + \mathbbm{1}_{\tau > T} \left(\epsilon +\int_{\mathbb{R}^d} p_\rho(x) \|x\| dx \right) \\
	& \leq \max \left\{C, \epsilon +\int_{\mathbb{R}^d} p_\rho(x) \|x\| dx \right\}.
	\end{align*}
	Consequently, there exists a constant $D<\infty$ such that for any $\tau>0$
	\begin{equation*}
	\mathbb{E}^{\mathbb{P}_\beta}[\mathcal{W}_f(\mu,\Psi(\tau,\theta_{-(t+\tau)}\beta) \mu)] < D,
	\end{equation*}
	hence $\mathcal{W}_f(\mu,\Psi(\tau,\theta_{-(t+\tau)}\beta) \mu)< D$ for $\mathbb{P}_B$-almost all $\beta \in \Omega_B$. This shows that
	\begin{equation*}
	\lim_{t \to \infty} \mathcal{W}_f\left(\Psi(t, \theta_{-t} \beta) \mu,\Psi(t+\tau, \theta_{-(t+\tau)} \beta) \mu \right) = 0
	\end{equation*}
	for all $\tau>0$ and $\mathbb{P}_B$-almost all $\beta \in \Omega_B$. Therefore, $(\Psi(t,\theta_{-t}\beta)\mu)_{t>0}$ is $\mathbb{P}_B$-almost surely a Cauchy sequence with respect to the $W_f$ metric. Let $p_{t,\beta}$ denote the Lebesgue density of $\Psi(t,\theta_{-t}\beta)\mu$. Thanks to Proposition~\ref{prop:convergence_Wf_L1},
	\begin{equation*}
	\lim_{t \to \infty} p_{t,\beta} = p_\beta \in L^1
	\end{equation*}
	for all $t>0$ and $\mathbb{P}_B$-almost all $\beta \in \Omega_B$. Finally, we remark that the limit point $p_\beta \in L^1$ is invariant under the pullback flow. We have
	\begin{equation*}
	p_\beta = \Phi(t,\beta) p_{\theta_{-t}\beta}
	\end{equation*}
	for all $t>0$. In light of the discussion in Appendix~\ref{app:non_auto_FP_IVP_proof}, $\Phi$ maps $L^1$ functions to $\mathcal{S}$ functions. In other words, $p_\beta \in \mathcal{S}$. We conclude that for $\mathbb{P}_B$-almost all $\beta \in \Omega_B$ and initial probability density $p \in L^1$, there exists a unique pullback attractor for~\eqref{eq:non autonomous FP gradient}
	\begin{equation*}
	p_\beta = \lim_{t \to \infty} \Phi(t,\theta_{-t} \beta) p \in \mathcal{S}.
	\end{equation*}
\end{proof}

%------------------------------------- Proof: Prop disintegration -------------------------------------%
\begin{proof}[Proof of Proposition~\ref{prop:cnpbversusdisintegration}]
	By construction, for every $C\in\mathscr{B}(\mathbb{R}^d)$, any Borel measure $\nu$ on $\mathbb{R}^d$ with Lebesgue density $p_\nu\in\mathcal{S}$ and for $\mathbb{P}_B$-almost all $\beta\in\Omega_B$, we have for all $s\leq t$
	\begin{equation}\label{eq:fun_Phi}
	\nu_t(C):=\int_{\Omega_W} \phi_*(t-s;\theta_s \omega, \theta_s \beta)\nu(C) \mathbb{P}_W(d\omega) = \int_C \Phi(t-s;\theta_s \beta) p_{\nu_s} dx=
	\int_Cp_{\nu_t}dx,
	\end{equation}
	and, using the results from proposition~\ref{prop:disintegration} and \eqref{eq:fun_Phi},
	\begin{equation*}
	\begin{split}
	\mu_\beta(C) &= \int_{\Omega_W} \mu_{\omega,\beta}(C)\mathbb{P}_W(d\omega)  
	= \int_{\Omega_W} \lim_{\tau \to \infty} \phi(\tau,\theta_{-\tau} \omega, \theta_{-\tau} \beta)_* \rho(C) \mathbb{P}_W(d\omega) \\
	& = \lim_{\tau \to \infty} \int_{\Omega_W} \phi(\tau,\theta_{-\tau} \omega, \theta_{-\tau} \beta)_* \rho(C)\mathbb{P}_W(d\omega) 
	= \lim_{\tau \to \infty} \int_C \Phi(\tau, \theta_{-\tau} \beta) p_\rho dx \\
	& = \int_C \lim_{\tau \to \infty} \Phi(\tau, \theta_{-\tau} \beta) p_\rho dx
	\end{split}
	\end{equation*}
	by which the result follows from the fact that $\mu_\beta(C)=\int_Cp_\beta dx$, for all $C\in\mathscr{B}(\mathbb{R}^d)$.
\end{proof}

%------------------------------------- Proof: Prop ergodicity ---------------------------------------%
\begin{proof}[Proof of Proposition~\ref{prop:erg}]
The Dirac measure $\delta_{p_\beta}$ is the disintegration of a Markov measure
of the random dynamical system $\Phi$ on $\Omega_B\times\mathcal{S}$, associated to the stochastic Fokker-Planck \eqref{stochasticFP}. The $\mathcal{S}$-marginal of this Markov measure
\begin{equation}\label{StatStochFP}
P:=\int_{\Omega_B} \delta_{p_\beta}\mathbb{P}_B(d\beta)
\end{equation}
is the corresponding stationary measure of \eqref{stochasticFP}.  
Application of Birkhoff's Ergodic Theorem then yields that time-averages of ($P$-integrable) observables $g:\mathcal{S}\to\mathbb{R}$ satisfy
\[
\lim_{\tau\to\infty}\frac{1}{\tau}\int_0^\tau g(\Phi(\tau,\beta) p) dt=\int_\mathcal{A}  g(p) P(dp)=\int_{\Omega_B} g(p_\beta)\mathbb{P}_B(d\beta),
\]
$\mathbb{P}_B\times P$-almost surely. By Elton's Ergodic Theorem  \cite{elton1987ergodic}, this relation holds in fact 
$\mathbb{P}_B$-almost surely if $g$ is continuous. 
\end{proof} 

%-------------------------------------------------------------------------------------------------------------
%					Explicit disintegration and pullback attractor in the linear case
%-------------------------------------------------------------------------------------------------------------
\section{Exact solutions for the Ornstein-Uhlenbeck SDE with intrinsic and common additive noise} \label{app:explicit_disintegration}

In this appendix, we present the calculations of the closed expressions for common noise pullback attractors of 
the one-dimensional Ornstein-Uhlenbeck SDE with additive intrinsic and common noise, discussed in Section~\ref{subsubsec:OU}. The flow of the SDE~\eqref{linSDE} from time $s$ to $t$ for fixed noise realisations $\beta$ and $\omega$ is explicitly given by
\begin{equation*}
\phi(t-s,\theta_s \omega, \theta_s \beta) x(s) = x(s) e^{-a(t-s)} + \eta \int_s^t e^{-a(t-u)} d \beta(u) + \sigma \int_s^t e^{-a (t-u)} d \omega(u).
\end{equation*}
Averaging this equation over the intrinsic noise yields 
\begin{equation}\label{eq:average_intrinsic_noise}
\begin{split}
\int_{\Omega_W} \phi(t-s,\theta_s \omega, \theta_s \beta) x(s) \mathbb{P}_W(d \omega) & = x(s) e^{-a(t-s)} + \eta \int_s^t e^{-a(t-u)} d \beta(u) \\
\quad &  + \sigma \int_s^t e^{-a (t-u)} d W(u).
\end{split}
\end{equation}
It's important to emphasize that the integral with respect to the single path $\beta$ is a real number while the integral with respect to the intrinsic noise $W$ is a Gaussian distribution. The density of the distribution in~\eqref{eq:average_intrinsic_noise} is given by 
\begin{equation}
p(x,t) = \sqrt{\frac{a}{\pi \sigma^2 (1-e^{-2 a (t-s)})}} \exp \left( - \frac{a}{\sigma^2 (1-e^{-2 a (t-s)})} (x-m_\beta(t,s))^2 \right),
\label{trans_prob_p_beta}
\end{equation}
where
\begin{equation*}
m_\beta(t,s) := x(s) e^{-a(t-s)} + \eta \int_s^t e^{-a (t-u)} d \beta(u).
\end{equation*}
It is readily checked that indeed the density~\eqref{trans_prob_p_beta} is a solution of the stochastic Fokker-Planck equation ~\eqref{stochasticFP} with $V(x) = \frac{a}{2} x^2$.

Averaging \eqref{eq:average_intrinsic_noise} over the common noise yields
\begin{equation}
\begin{split}
\int_{\Omega_B} \int_{\Omega_W} \phi(t-s,\theta_s \omega,\theta_s \beta) x(s) \mathbb{P}_W(d \omega) \mathbb{P}_B(d \beta) & = x(s) e^{-a(t-s)} + \eta \int_s^t e^{-a(t-u)} d B(u) \\
\quad &  + \sigma \int_s^t e^{-a (t-u)} d W(u),
\end{split}\label{doubav}
\end{equation}
where now both integrals represent Gaussian distributions. The density of the distribution \eqref{doubav} is
\begin{equation}
\bar{p}(x,t) = \int_{\Omega_B} p(x,t) \mathbb{P}_B(d \beta)=
\sqrt{\frac{a}{\pi (\eta^2 + \sigma^2) (1-e^{-2 a (t-s)})}} \exp \left\{- \frac{a}{(\eta^2 + \sigma^2)(1-e^{-2 a(t-s)})} x^2 \right\},
\label{density_p_classic_FP}
\end{equation}
which in turn is a solution of the Fokker-Planck equation of the SDE \eqref{SDE} with $V(x) = \frac{a}{2} x^2$:
\begin{equation*}
\frac{\partial \bar{p}}{\partial t} = a \bar{p} + a x \frac{\partial \bar{p}}{\partial x} + \frac{1}{2} (\sigma^2 + \eta^2) \frac{\partial^2 \bar{p}}{\partial x^2}.
\end{equation*}
We conclude this section with a discussion on pullback attractors. The pullback attractor of the SDE~\eqref{linSDE} with respect to both intrinsic and common noise is
\begin{equation}
\alpha(\omega,\beta) := \lim_{s \to -\infty} \phi(t-s,\theta_s \omega, \theta_s \beta) x(s) = \eta \int_{-\infty}^0 e^{au} d \beta(u) + \sigma \int_{-\infty}^0 e^{a u} d \omega(u).\label{FP}
\end{equation}
This is a point attractor, confirming that, at the SDE level, the system is synchronizing. The fiberwise measures resulting from disintegration (see Section~\ref{sec:from_nonauto_to_stoch}) are therefore
\begin{equation*}
\mu_{\omega,\beta} = \delta_{\alpha(\omega,\beta)}.
\end{equation*}
Integrating with respect to the intrinsic noise yields 
\begin{equation*}
\mu_\beta := \int_{\Omega_W} \mu_{\omega,\beta} \mathbb{P}_W(d \omega) = \eta \int_{-\infty}^0 e^{au} d \beta(u) + \sigma \int_{-\infty}^0 e^{a u} d W(u).
\end{equation*}
This is normally distributed with variance depending on the intensity of the intrinsic noise $\sigma$ and mean depending on the intensity of the common noise $\eta$. Its density is
\begin{equation}
p_\beta(x) = \sqrt{\frac{a}{\pi \sigma^2}} \exp \left\{-\frac{a}{\sigma^2} \left(x-\eta \int_{-\infty}^0 e^{au} d \beta(u) \right)^2 \right\}.
\label{density_pullback}
\end{equation}
Finally, integrating over all common noise realizations we obtain the stationary measure
\begin{equation*}
\begin{split}
\rho & = \int_{\Omega_B} \mu_\beta \mathbb{P}_B(d \beta) = \int_{\Omega_B} \int_{\Omega_W} \mu_{\omega,\beta} \mathbb{P}_W(d \omega) \mathbb{P}_B(d \beta) \\
\quad & = \eta \int_{-\infty}^0 e^{au} d B(u) + \sigma \int_{-\infty}^0 e^{a u} d W(u),
\end{split}
\end{equation*}
with density
\begin{equation*}
p_\rho(x) = \sqrt{\frac{a}{\pi (\sigma^2 + \eta^2)}} \exp \left\{-\frac{a}{(\eta^2 + \sigma^2)} x^2 \right\}.
\end{equation*}
Of course, we also have $p_\rho(x)=\int_{\Omega_B} p_\beta(x) \mathbb{P}_B(d\beta)$ and
$p_\rho(x)=\lim_{(t-s)\to\infty}\bar{p}(x,t)$ confirming global convergence of solutions of \eqref{FP} to the stationary measure
in forward and pullback sense.

%%%%%%%%%%%%%%%%%%%%%%%%%%%%%%%%%%%%%%%%%%%%%%%%%%%%%%%%%%%%%%%%%%%%%%%%%%%%%%%%%%%%%%%%%%%%%%%%%%%%%%%%%%%%%%
%											Acknowledgements
%%%%%%%%%%%%%%%%%%%%%%%%%%%%%%%%%%%%%%%%%%%%%%%%%%%%%%%%%%%%%%%%%%%%%%%%%%%%%%%%%%%%%%%%%%%%%%%%%%%%%%%%%%%%%%
\section*{Acknowledgements}\label{sec:acknow}
We are grateful to Sergey Zelik for extended discussions and input on the material in Section~\ref{subsec:exist_uniq_reg} and Appendix~\ref{app:non_auto_FP_IVP_proof} and to Martin Hairer for pointing out \cite{eberle2016reflection}. FG has been supported by the EPSRC Centre for Doctoral Training in Financial Computing and Analytics.  JSWL acknowledges support by the UK Royal Society (NAF\textbackslash{R}1\textbackslash180236), the EU H2020 ITN CRITICS (643073) and the London Mathematical Laboratory (LML) through its External Fellowship scheme.

%%%%%%%%%%%%%%%%%%%%%%%%%%%%%%%%%%%%%%%%%%%%%%%%%%%%%%%%%%%%%%%%%%%%%%%%%%%%%%%%%%%%%%%%%%%%%%%%%%%%%%%%%%%%%%
%												References
%%%%%%%%%%%%%%%%%%%%%%%%%%%%%%%%%%%%%%%%%%%%%%%%%%%%%%%%%%%%%%%%%%%%%%%%%%%%%%%%%%%%%%%%%%%%%%%%%%%%%%%%%%%%%%

\bibliography{references}
\bibliographystyle{siam}

\end{document}